\documentclass[notitlepage]{amsart}
\usepackage{amssymb,amsfonts,latexsym, esint}
\usepackage{graphics,verbatim}
\usepackage{graphicx}

\newtheorem{theorem}{Theorem}[section]
\newtheorem{proposition}{Proposition}[section]
\newtheorem{lemma}{Lemma}[section]

\newtheorem{remark}{Remark}[section]
\newtheorem{definition}{Definition}[section]

\newcommand{\eps}{\varepsilon}
\newcommand{\deb}{\rightharpoonup}

\newcommand{\To}{\longrightarrow}

\newcommand{\norm}[1]{\left\Vert#1\right\Vert}

\newcommand{\be} {\begin{equation}}
\newcommand{\ee} {\end{equation}}
\newcommand{\bea} {\begin{eqnarray}}
\newcommand{\eea} {\end{eqnarray}}
\newcommand{\Bea} {\begin{eqnarray*}}
\newcommand{\Eea} {\end{eqnarray*}}
\newcommand{\pa} {\partial}

\newcommand{\al} {\alpha}
\newcommand{\ba} {\beta}
\newcommand{\de} {\delta}
\newcommand{\na}{\nabla}
\newcommand{\ga} {\gamma}
\newcommand{\Ga} {\Gamma}
\newcommand{\Om} {\Omega}
\newcommand{\om} {\omega}
\newcommand{\De} {\Delta}
\newcommand{\la} {\lambda}

\newcommand{\no} {\nonumber}
\newcommand{\noi} {\noindent}
\newcommand{\lab} {\label}
\newcommand{\va} {\varphi}
\newcommand{\f}{\frac}
\newcommand{\R}{\mathbb R}

\newcommand{\Rn}{\mathbb R^N}
\newcommand{\Iom}{\int_{\Omega}}
\catcode`\@=11

\makeatletter \@addtoreset{equation}{section} \makeatother

\begin{document}

\title[Profile of solutions of nonlocal equations]{Profile of solutions for nonlocal equations with critical and supercritical nonlinearities}

\author{Mousomi Bhakta, Debangana Mukherjee}
\address{M. Bhakta, Department of Mathematics, Indian Institute of Science Education and Research, Dr. Homi Bhaba Road, Pune-411008, India}
\email{mousomi@iiserpune.ac.in}
\address{D. Mukherjee, Department of Mathematics, Indian Institute of Science Education and Research, Dr. Homi Bhaba Road, Pune-411008, India}
\email{debangana18@gmail.com}
\author{ Sanjiban Santra}
\address{S. Santra, Department of Basic Mathematics\\
Centro de Investigaci\'one en Mathematic\'as \\ Guanajuato,
M\'exico}
\email{sanjibansntr385@gmail.com}
\thanks{}

\date{}

\begin{abstract} We study the  fractional Laplacian problem
\begin{equation*}
(I_{\eps})  \quad\,\,\,\,\,
\left\{\begin{aligned}
      (-\De)^s u &=u^p -\eps u^q \quad\text{in }\quad \Om, \\
      u &>0  \quad\text{in }\quad \Om,\\
       u&=0 \quad\text{in}\quad \Rn\setminus\Om,\\
       u &\in H^s(\Om)\cap L^{q+1}(\Om); \\
          \end{aligned}
  \right.
\end{equation*}
where $s\in(0,1)$,  $q>p\geq \f{N+2s}{N-2s}$ and $\eps>0$ is a parameter. Here
$\Om\subseteq\Rn$ is a bounded star-shaped domain with smooth boundary and  $N> 2 s$. We establish existence of a variational positive solution $u_{\eps}$ and characterise the asymptotic behaviour of $u_{\eps}$ as $\eps\to 0$. When $p=\f{N+2s}{N-2s}$, we describe how the solution $u_{\eps}$ blows up at a  interior point of $\Omega$. Furthermore, we prove  the {\it local uniqueness} of solution of the above problem when $\Om$ is a convex symmetric  domain of $\Rn$ with $N>4s$ and $p=\f{N+2s}{N-2s}$.
 \end{abstract}

\maketitle
\section{\bf Introduction}
There has been considerable interest in understanding the asymptotic behavior of
positive solutions of the elliptic problem
\begin{equation}
  \label{a1}
\left\{\begin{aligned}
      \eps^{2s} (-\De)^s u &= f(u)   \quad\text{in } \Om,\\
    u &>0 \quad \text{in} \ \Omega,  \\
      u &= 0 \quad\text{on } \R^N \setminus  \Om,
    \end{aligned}
  \right.
\end{equation}
where $\eps>0$ is a parameter, $s\in(0,1)$ and $f$ is having superlinear nonlinearity with $f(0)=0$.
$\Om $ is a smooth bounded domain in $\mathbb{R}^N$. The existence and asymptotic behavior of solutions to (\ref{a1}) depend crucially on the behavior of $f$ near $0$.  It is easy to check that  problem (\ref{a1})  may not have any nontrivial solutions for small $\eps>0$ if $f'(0) > 0$.  The case of $ f' (0) <0$ has been studied by many authors. To mention a few of them in the local case, we refer the papers  \cite{DF}, \cite{NW} and the references therein.  In the nonlocal case, not much is known. Multi-peak solutions of a fractional Schr\"odinger equation in the whole of $\R^N$ was considered  in \cite{DPW}. In \cite{DPDV},  D\'{a}vila,  et al  constructed a family of solutions which have the properties that,  when $\eps \to 0$, those solutions  concentrate at an interior point of the domain in the form of a scaling ground state in entire space. Bubble solutions for the fractional problems involving the almost critical or almost supercritical powers were considered in D\'{a}vila  et al et al \cite{DRS}. 

In this paper, we consider the problem in the {\it zero
mass case} i.e., when $f(0)=0$ and $f'(0)=0.$    The problem (\ref{a1}) can be viewed as {\it borderline} problems. When $s=1,$ Berestycki and Lions in \cite{BL} proved the existence of ground
state solutions if $f(u)$ behaves like $|u|^p$ for large $u$ and
$|u|^q$ for small $u$ where $p$ and $q$ are respectively
supercritical and subcritical.

\vspace{2mm}

In this paper,  we consider the following family of problems:
\begin{equation}
  \label{eq:a3'}
\left\{\begin{aligned}
      (-\De)^s u &=u^p -\eps u^q \quad\text{in }\quad \Om, \\
      u &>0 \quad\text{in }\quad \Om,\\
            u &=0 \quad\text{in}\quad\Rn\setminus\Om,\\
            u &\in H^s(\Om)\cap L^{q+1}(\Om), 
 \end{aligned}
  \right.
\end{equation}
where $s\in(0,1)$ is fixed, $(-\De)^s$ denotes the  fractional Laplace operator defined,  up to a normalisation factor, as
\begin{align} \label{De-u}
  -\left(-\Delta\right)^s u(x)=\frac{1}{2}\int_{\mathbb{R}^N}\frac{u(x+y)+u(x-y)-2u(x)}{|y|^{N+2s}}dy, \quad x\in\Rn.
\end{align}
In \eqref{eq:a3'}, $q>p\geq 2^*-1=\f{N+2s}{N-2s}$, $\eps>0$ is a parameter,
$\Om\subseteq\Rn$ is a bounded star-shaped domain with smooth boundary and $N> 2 s$. Note under a suitable change of variable \eqref{eq:a3'} can be transformed in the form of \eqref{a1}.

We denote by $H^s(\Om)$, the usual fractional Sobolev space endowed with the so-called Gagliardo norm
\be\lab{norm-H}\norm{g}_{H^s(\Om)}=\norm{g}_{L^2(\Om)}+\bigg(\int_{\Om \times \Om} \frac{|g(x)-g(y)|^2}{|x-y|^{N+2s}}dxdy\bigg)^{1/2}.\ee
For further details on the fractional Sobolev spaces we refer to \cite{NPV} and the references therein.  Note that, in problem \eqref{eq:a3'} 
the Dirichlet datum is given in  $\Rn\setminus\Om$ and not simply
 on $\pa\Om$ and therefore we need to introduce a new functional space $X_0$, which, in our opinion, is the suitable space to work with.   
\be\label{eq:X0}
X_0(\Om):=\{v\in H^s(\Rn): v=0\quad \text{in}\quad \Rn\setminus\Om\}.
\ee
By \cite[Lemma 6 and 7]{SerVal3}, it follows that 
\be\label{norm-X}
||v||_{X_0}=\displaystyle\left(\int_{Q} \f{|v(x)-v(y)|^2}{|x-y|^{N+2s}}dxdy,\right)^\f{1}{2},
\ee
where $Q=\R^{2N}\setminus (\Om^c\times\Om^c)$ is a norm on $X_0$ and $(X_0, ||.||_{X_0})$ is a Hilbert space, with the inner product 
$$<u, v>_{X_0}=\int_{Q} \f{(u(x)-u(y))(v(x)-v(y))}{|x-y|^{N+2s}}dxdy.$$
We observe that, norms in \eqref{norm-H} and \eqref{norm-X} are not same in general, since $\Om\times\Om$ is strictly contained in $Q$ (see \cite{ SerVal2, SerVal3}) but \eqref{norm-H} and \eqref{norm-X} are equivalent in some cases, such as $s>1/2$. Clearly, the integral in (1.6) can be extended to whole of $\R^{2N}$ as $v=0$ in $\R^{N}\setminus\Om$.
It follows from \cite[Lemma 8]{SerVal3} that the embedding $X_0\hookrightarrow L^r(\Rn)$ is compact, for any $r\in[1, 2^*)$  and from \cite[Lemma 9]{SerVal2} that $X_0\hookrightarrow L^{2^*}(\Rn)$ is continuous.

\begin{definition}\lab{def-1}
We say that $u\in X_0\cap L^{q+1}(\Om)$ is a weak solution of Eq. \eqref{eq:a3'}, if $u>0$ in $\Om$ and for every $\va\in X_0$, 
$$
\int_{\Rn}\int_{\Rn}\f{(u(x)-u(y))(\va(x)-\va(y))}{|x-y|^{N+2s}}dxdy=\int_{\Om}u^p\va\ dx-\eps\int_{\Om}u^q\va\ dx.$$
\end{definition}

\vspace{3mm}

In recent years, a great deal of attention has been devoted to equations  of elliptic/parabolic type with fractional and non-local operators because these kind of equations play  important role in the real world and many perfect techniques which have been developed by well-known mathematicians during the past decades can not be directly applied to the fractional case. These equations arise from models in physics, engineering (see \cite{MK}), optimisation and finance (see \cite{CT}), obstacle problem (see \cite{S}), conformal geometry and minimal surface (see \cite{CG}) and many more, see for instance, \cite{A, BCSS, V} and the references therein. 

Nonlinear nonlocal problems of the form $(-\De)^s u=f(u)$ were 
studied by many authors where $f:\Rn\to\R$ is a certain function. Since it is almost impossible to describe all the works involving them,we explain only few of them, which are related to our problem.  In \cite{SerVal2}, Servadei and Valdinoci  studied the Brezis-Nirenberg problem in the nonlocal case. More precisely, they considered the nonlinearity of the form $\la u+u^{2^*-1}$, with $\la>0$. On the other hand, in \cite{SerVal3} the same authors studied mountain-pass solutions for the equation with general integro-differential operator and with the nonlinearities of subcritical growth.
In \cite{BM}, first and second authors 
of this paper studied the equation in whole of $\Rn$ with nonlinearities involving critical and supercritical growth. They established decay estimate of solution and the gradient of the solution at infinity and using that they prove nonexistence result via Pohozaev identity.

In the local case, $s=1$, Merle and Peletier \cite{MP} considered the equation \eqref{eq:a3'}. They proved that for $N \geq 3$, problem  \eqref{eq:a3'} possesses a family of solutions concentrating at a point $\xi_{0}$, which is a critical point of the Robin function $R$. In this paper we extend the result to the fractional Laplacian case.

\vspace{3mm}

For the supercritical case ($p>2^*-1$), define, 
\be\label{a21} F(u, \Om)=  \frac{1}{2}\frac{\displaystyle\int_{\Rn\times\Rn}\f{|u(x)-u(y)|^2}{|x-y|^{N+2s}}dx dy}{\displaystyle\int_{\Omega} |u|^{p+1} dx}+\frac{1}{q+1}\frac{\displaystyle\int_{\Om}|u|^{q+1} dx}{\displaystyle\left(\int_{\Omega} |u|^{p+1} dx\right)^l},\ee
where $l=\frac{2s(q+1)-N(p-1)}{2s(p+1)-N(p-1)}$,   $u\in X_0(\Om)\cap L^{q+1}(\Om)$ and
\be\lab{a2}
\mathcal{K}:= \inf\bigg\{F(u, \Rn): u\in D^{s,2}(\Rn)\cap L^{q+1}(\Rn),\,  \int_{\Rn}|u|^{p+1}=1\bigg\},
\ee
where $D^{s,2}(\Rn)$ is the closure of $C^{\infty}_0(\Rn)$ w.r.t. to the norm $\bigg(\displaystyle\int_{\Rn \times \Rn} \frac{|u(x)-u(y)|^2}{|x-y|^{N+2s}}dxdy\bigg)^{1/2}$.

For the critical case ($p=2^*-1$), we consider the usual functional
\be\lab{sep-17-1}
S(u)=\frac{\displaystyle\int_{\Rn\times\Rn}\f{|u(x)-u(y)|^2}{|x-y|^{N+2s}}dx dy}{\displaystyle\left(\int_{\Om}|u|^{p+1} dx\right)^\f{2}{p+1}},
\ee
where $u\in X_0(\Om)$.\\

Define, the Sobolev constant
\bea\label{go} \mathcal{S}:&=&  \inf_{u\in D^{s,2}(\Rn)\setminus\{0\}}\frac{\displaystyle\int_{\Rn\times\Rn}\f{|u(x)-u(y)|^2}{|x-y|^{N+2s}}dxdy}{\displaystyle\left(\int_{\Rn} |u|^{2^*} dx\right)^\f{2}{2^*}}\eea
or, equivalently, $$\mathcal{S}=\inf\bigg\{S(v): v\in D^{s,2}(\Rn), \ \int_{\Rn}|v|^{2^*}dx=1\bigg\}.$$
It is well known by \cite{L} that  $\mathcal{S}$ is achieved by 
\be\label{ent-U}
U(x)=c_{N,s}\big(1+|x|^2\big)^{-\big(\f{N-2s}{2}\big)},
\ee
where \be\lab{c-ns}
c_{N,s}=2^\f{N-2s}{2}\displaystyle\left(\f{\Ga(\f{N+2s}{2})}{\Ga(\f{N-2s}{2})}\right)^\f{N-2s}{4s}.\ee
By \cite{CLO} and \cite{LZ},  a direct computation implies that for all $\eps >0$ and for any $a\in \R^N$,  $U$ is the unique solution
satisfying $$U_{\eps, a}(x)=\eps^{-\frac{N-2s}{2}}U\bigg(\frac{x-a}{\eps}\bigg)
$$
and verifies the following equation
\begin{equation}
  \label{ent}
\left\{\begin{aligned}
      (-\Delta)^s U&=U^{2^*-1} \quad\text{in }\quad \R^N,\\ U&>0  \quad\text{in }\quad \R^N, \\
      U &\in D^{s,2}(\Rn).
         \end{aligned}
  \right.
\end{equation}
Define the Green's function $G=G(x,y)$ of the operator $(-\De)^s$ in $\Om$ for $x, y\in\Om$ as  
\begin{equation}
  \label{green}
  \left\{
    \begin{aligned}
      (-\De_x)^s G(x, y)&=  \de_{y} \; &&\text{in } \Om, \\
G(x, y) &=0  \; &&\text{in } \Rn\setminus \Om.
    \end{aligned}
  \right.
\end{equation}
It is convenient to introduce the regular part of $G$, which is often denoted by $H$, defined by
\be\lab{eq:gr-H}
 G(x, y):=F(x, y)-H(x,y),
\ee
where the function $H$ satisfies
\begin{equation}
  \label{rob}
  \left\{
    \begin{aligned}
      (-\De_x)^s H(x, y)&=  0 \; &&\text{in } \Om, \\
H(x, y) &=F(x, y)  \; &&\text{in } \Rn\setminus \Om,
    \end{aligned}
  \right.
\end{equation}
for any fixed $y\in \Om$
and
\be\lab{eq:gr-F} F(x, y)= \frac{a_{N,s}}{|x-y|^{N-2s}} ,\ee
 is the fundamental solution of the  elliptic operator $(-\De)^s.$  In \eqref{eq:gr-F},  $a_{N,s}$ is 
 $$a_{N,s}:=\f{\Ga(\f{N}{2}-s)}{2^{2s}\pi^\f{N}{2}\Ga(s)},$$ (see \cite{B}).  Define the Robin function as
\begin{equation}\label{Robin}
R(x)= H(x, x).
\end{equation}
For the continuity of $R$, see Abatangelo \cite{A1}. 

\begin{definition}
We say $\Om$ is strictly star-shaped with respect to the point $y$, if 
$$\big\langle x-y, n(x)\big\rangle>0 \quad\forall\quad x\in\pa\Om,$$
where $n(x)$ is the unit outward normal to $\pa\Om$ at $x$.  
\end{definition}

We recall here the general Pohozaev identity in the nonlocal case due to Ros-Oton and Serra \cite{RS}:
Let $u$ be a bounded solution of 
\begin{equation}
  \label{eq:RS-1}
\left\{\begin{aligned}
      (-\De)^s u &=f(u) \quad\text{in }\quad \Om, \\
         u &=0 \quad\text{in}\quad\Rn\setminus\Om,
 \end{aligned}
  \right.
\end{equation}
where $\Om\subset\Rn$ is a bounded $C^{1,1}$ domain, $f$ is locally Lipschitz and $d(x)=\text{dist}(x,\pa\Om)$. Then $u$ satisfies the following identity:
\be\lab{poho-1}
(2s-N)\int_{\Om}uf(u)\ dx+2N\int_{\Om}F(u)\ dx=\Ga(1+s)^2\int_{\pa\Om}\bigg(\f{u(x)}{d^s(x)}\bigg)^2\big\langle x,\,\nu(x)\big\rangle dS(x),
\ee
 where $F(t)=\displaystyle\int_0^t f(s)ds,$ \, $\nu(x)$ is the unit outward normal to $\pa\Om$ at $x$ and $\Ga$ is the Gamma function.

Translating the function $u$, it is easy to see that, when $\Om$ is a  $C^{1,1}$ bounded domain,  the following general identity holds:
\be\lab{poho}
(2s-N)\int_{\Om}uf(u)\ dx+2N\int_{\Om}F(u)\ dx=\Ga(1+s)^2\int_{\pa\Om}\bigg(\f{u(x)}{d^s(x)}\bigg)^2\big\langle x-y,\,\nu(x)\big\rangle dS(x),
\ee
for every $y\in\Rn$.

Note that, by the above Pohozaev identity \eqref{eq:a3'} does not have any solution in a star-shaped domain when $\eps=0$.

\vspace{3mm}

We turn now to a brief description of the results presented below. \begin{theorem}\label{main1}
There exists $\eps_n>0$ and $\la_n>0$ with $\eps_n\to 0$ as $n\to\infty$ and $\la_n$ uniformly bounded above and away from zero, such that
\begin{itemize}
\item[(i)] there exists a solution $u_n$ to Eq. \eqref{eq:a3'} corresponding to $\eps=\eps_n$;
\item[(ii)] if $p>2^*-1$, then $F(\la_n u_n)\to \mathcal{K}$ and $\displaystyle\int_{\Om} u_n^{p+1}dx\to 0$ as $n\to\infty$;
\item[(iii)]if $p=2^*-1$, then $S(u_n)\to\mathcal{S}$ as $n\to\infty$ and there exist  constants $A, B>0$ such that for all $n\geq 1$, it holds $A<\displaystyle\int_{\Om}  u_n^{p+1}dx< B$ ,
\end{itemize}
where $F(.)$,\, $S(.)$,\, $\mathcal{K}$ and $\mathcal{S}$ are defined as in \eqref{a21}, \eqref{sep-17-1}, \eqref{a2} and \eqref{go} respectively.
\end{theorem}

\begin{theorem}\label{main2}  Let $\Om$ be a smooth bounded star-shaped domain with respect to 0, $2^*-1=p<q$.  Suppose  $u_{\eps}\in X_0( \Om)$ is a  solution  of Eq. (\ref{eq:a3'})  such that
\be\lab{nov-1-1}S(u_{\eps}) \to \mathcal{S} \quad\text{and}\quad
 A<\displaystyle\int_{\Om}u_{\eps}^{p+1}dx< B,\ee where $S(.)$, $\mathcal{S}$ are as in \eqref{sep-17-1} and \eqref{go} respectively.  
Let $x_{\eps}$ be a point such that $||u_{\eps}||_{L^{\infty}}=u_{\eps}(x_{\eps})$ Assume that, up to a  subsequence $x_{\eps}\to x_0$ as $\eps\to 0$. Then  $x_0$ is an interior point of $\Om$ and along a subsequence
\begin{eqnarray*}  \lim_{\eps \rightarrow 0}\eps \|u_{\eps}\|_{\infty}^{q-p+2}
=\f{\om_Nc_{N,s}^{2^*}}{2}\f{(q+1)R_{N, s, x_0}}{q(N-2s)-(N+2s)} s^2 \Gamma(s)^2B\bigg(\f{N}{2},\ s\bigg)^2\times\no\\
B\bigg(\f{N}{2},\ \big(\f{N-2s}{2}\big)q-s\bigg)^{-1},\no
\end{eqnarray*}
where  $c_{N, s}$ is defined in \eqref{c-ns} and  $B(a,b)$ is the Beta function defined by
\begin{equation}\label{eq:beta}
B(a,b)=\int_0^{\infty}t^{a-1}(1+t)^{-a-b}.
\end{equation}
Here $$R_{N, s, x_0}= \int_{\pa \Om} \bigg(\f{G(x, x_0)}{d^s(x)}\bigg)^2 \langle x-x_{0}, \nu \rangle dS.$$
Furthermore,
 \begin{equation} \label{nsd2} \lim_{\eps \rightarrow 0} \f{ u_{\eps}(x) \|u_{\eps}\|_{\infty} }{d^s(x)}= \f{\om_{N}c_{N,s}^{2^*}}{2}\frac{\Ga(\frac{N}{2})\Ga(s)}{\Gamma(\frac{N+2s}{2})}\f{G(x, x_0)}{d^s(x)}  \quad\text{in}\ C_{loc}(\overline{\Om}\setminus\{x_0\}),
\end{equation}
where $G(x,x_0)$ is the Green function as defined in \eqref{green} and $d(x)=\text{dist}(x,\pa\Om)$.
\end{theorem}

\vspace{3mm}

\begin{remark} Under a suitable modification to the Theorem \ref{main2}, a similar blow-up type result for the equation with $(-\De)^s$ operator in a smooth bounded domain $\Om$ with \it{outside} zero Dirichlet boundary condition  can be obtained for the nonlinearity $f_{1}(u)= u^{2^*-1-\eps}$ under the assumption
$$
\tilde{F}(u_{\eps}):=\frac{\displaystyle\int_{\Rn\times\Rn}\f{|u_{\eps}(x)-u_{\eps}(y)|^2}{|x-y|^{N+2s}}dx dy}{\displaystyle\left(\int_{\Om}|u|^{2^*-\eps} dx\right)^\f{2}{2^*-\eps}} \to \mathcal{S}  \text{ whenever } N> 2s$$ 
and for the nonlinearity $f_{2}(u)= u^{2^*-1}+ \eps u$ under the assumption
\be\no S(u_{\eps}) \to \mathcal{S} \text{ whenever } N> 4s .\ee
\end{remark}

Concerning the \textit{uniqueness} problem, the shape of the domain plays an important role and hence some assumptions on $\Om$ is needed, see \cite{G}. To prove uniqueness theorem, our assumption on the domain are the following:

\vspace{2mm}

(A1) $\Om$ is symmetric with respect to the hyperplanes $\{x_i=0\}, i=1, 2, \cdots, N$.

(A2) $\Om$ is convex in the $x_i$ directions, $i=1, 2 \cdots, N$.

\begin{remark}\lab{r:nov-1} By (A1), (A2) and in virtue \cite[Theorem 3.1]{FW} (also see  \cite[Corollary 1.2]{JW}), every solution $u_{\eps}$ of \eqref{eq:a3'} is symmetric with respect to the hyperplanes  $\{x_i=0\}, i=1,\cdots, N$
and strictly decreasing in the $x_i$ direction,  $\ i=1,\cdots, N$ . Therefore $$\max_{x\in\Om}u_{\eps}(x)=u_{\eps}(0).$$
\end{remark}

\begin{theorem}\lab{t:unique}
Let $2^*-1=p<q$ and $\Om$ be smooth bounded  star-shaped domain in $\Rn$ with respect to $0$, $N>4s$, satisfying (A1) and (A2). Suppose $u_{\eps}$ and $v_{\eps}$ are two solutions of \eqref{eq:a3'} with $\max_{x\in\Om}u_{\eps}=\max_{x\in\Om}v_{\eps}$ and satisfy \eqref{nov-1-1} . Then, there exists $\eps_0>0$ such that $\forall \eps\in(0,\eps_0)$, 
$$u_{\eps}\equiv v_{\eps} \quad\text{in}\quad \Om.$$
\end{theorem}

The rest of the paper is organised as follows. In Section 2, we prove Theorem 1.1. Section 3 deals with the proof of Theorem 1.2. Section 4 is devoted to the study of  uniqueness result. The last section is the  Appendix. 
Laplace

\vspace{3mm}

{\bf Notations}: Throughout this paper $C$ denotes the generic constants which may vary from line to line. Below are few notations which we use throughout the paper:
\begin{itemize}
\item $\om_N=$ surface measure of unit ball in $\Rn$,
\item $G(x,y)$ denotes the Green function of $(-\De)^s$ in $\Om$,
\item $B(.,.)$ and  $\Ga(.)$ denote the Beta function and the Gamma function respectively.
\item $D^{s,2}(\Rn)$ denotes the closure of $C^{\infty}_0(\Rn)$ with respect to 
$(\int_{\Rn}\int_{\Rn}\frac{|u(x)-u(y)|^2}{|x-y|^{N+2s}}dxdy)^\frac{1}{2}$.
\end{itemize}

\section{\bf Asymptotic behavior}

\begin{proposition}\label{p:main1}
Let $2^*-1\leq p<q$. Then there exists $\eps_{0}>0$ such that  for  all $\eps\in (0, \eps_{0})$,  the problem
\begin{equation}
  \label{eq:eps_la}
\left\{\begin{aligned}
      (-\De)^s v&=\la_{\eps}v^p -\eps  v^q &&\text{in } \Om,\\ v&>0 && \text{ in } \Om, \\
      v(x) & =0 &&\text{ in } \Rn\setminus \Om,
          \end{aligned}
  \right.
\end{equation}
admits a solution $v_{\eps}$, with the property that
$$A<\la_{\eps}<B,$$ for some constants $A,B>0$, independent of $n$. In addition
\begin{itemize}
\item[(i)] if $p>2^*-1$, then $F(v_\eps)\to \mathcal{K}$ and $\displaystyle\int_{\Om}  v_{\eps}^{p+1}dx\to 0$ as $\eps\to 0$;
\item[(ii)]if $p=2^*-1$, then $S(v_{\eps})\to \mathcal{S}$ as $\eps\to 0$ and  $\displaystyle\int_{\Om}  v_{\eps}^{p+1}dx=1$,
\end{itemize}
where $\mathcal{K}$ and $\mathcal{S}$ are defined as in \eqref{a2} and \eqref{go} respectively.
\end{proposition}
\begin{proof}
Let $\Om_\eps=\frac{1}{\eps^{\frac{p-1}{2s(q-p)}}}\Om$ and $X_0(\Om_\eps)=\{w \in H^s(\Rn): w=0 \quad\text{in}\,\ \Rn \setminus \Om_\eps\}. $ Clearly $\Om_{\eps} \to \R^N$ as $\eps \to 0.$
Let us consider the manifold $N_\eps$ defined by:
$$N_\eps=\big\{w \in X_0(\Om_\eps)\cap L^{q+1}(\Om_\eps): \int_{\Om_\eps}w^{p+1}dx=1\big\}. $$
On $N_\eps, $ the functional $F$ can be written as:
\begin{eqnarray}\label{1}
 F(w)&=&\frac{1}{2}\int_{\Rn}\int_{\Rn}\frac{|w(x)-w(y)|^2}{|x-y|^{N+2s}}dxdy+\frac{1}{q+1}\int_{\Om_\eps}w^{q+1}dx\no\\
& =:& \hat{F}(w).
\end{eqnarray}
For $p\geq 2^*-1$, define
\begin{equation}\label{2}
 S_\eps:=\inf_{w \in N_\eps}\hat{F}(w)=\inf_{w \in N_\eps}F(w).
\end{equation}
Let $\{w_{n,\eps}\}\subset N_{\eps}$ be a minimizing sequence for \eqref{2}. Therefore, we have,
$$\hat{F}(w_{n,\eps}) \to S_\eps \,\text{as}\,\ n \to \infty, \int_{\Om_\eps}w_{n,\eps}^{p+1}dx=1. $$
Proceeding as in \cite[Theorem 1.5]{BM}, we can show that there exists $w_\eps \in X_0(\Om_\eps)\cap L^{q+1}(\Om_\eps)$ such that
$w_{n,\eps} \deb w_\eps$ in $X_0(\Om_\eps)$ and $w_{\eps}$ satisfies,

 $$(-\De)^s w_\eps=\la_{\eps}w_\eps^p -w_\eps^q \,\,\text{in }\,\, \Om_\eps \quad\text{and}\quad \hat{F}(w_{\eps})=S_{\eps}. $$

This yields,
$$\la_\eps=\int_{\Rn}\int_{\Rn}\frac{|w_{\eps}(x)-w_{\eps}(y)|^2}{|x-y|^{N+2s}}dxdy+\int_{\Om_\eps}w_\eps^{q+1}dx. $$
Since, $\hat{F}(w_\eps)=S_\eps$ we have, $2S_\eps <\la_\eps < (q+1)S_\eps. $ In Theorem A.1 (see Appendix), let $\rho=\eps^{\frac{-(p-1)}{2(q-1)}}, $
then $N_\rho$ and $S_{\rho}$ are exactly same as $N_\eps$ and $S_\eps$ defined here. Letting $\eps \to 0$ we have,
\begin{equation}\label{3}
 S_\eps \to \mathcal{K}\,\,\text{if}\,\ p>2^*-1,\,\ S_\eps \to \f{\mathcal{S}}{2}\,\,\text{if}\,\ p=2^*-1.
\end{equation}
Hence, there exists $\eps_0>0$ and $A,B>0$ such that $A< \la_\eps <B$ for all $\eps \in (0,\eps_0). $ Using the transformation 
$$v_{\eps}(x)= \eps^{-\frac{1}{(q-p)}} w_{\eps}(\eps^{-\frac{p-1}{2s(q-p)}}x),$$
we observe that $v_\eps$ is a solution of \eqref{eq:eps_la}. Moreover, $\displaystyle\int_{\Om_\eps}w_\eps^{p+1}dx=1$ implies  $\displaystyle\int_{\Om_\eps}v_\eps^{p+1}dx=
\eps^{\frac{p(N-2s)-(N+2s)}{2s(q-p)}}$.
Hence,

$$ \Iom v_\eps^{p+1}dx=1 \,\,\text{if}\,\,\ p=2^*-1$$ and
$$ \Iom v_\eps^{p+1}dx \to 0 \,\,\text{as}\,\,\ \eps \to 0, \, \, p> 2^*-1.$$
A simple calculation yields $$F(w_\eps)=\hat F(w_\eps)=F(v_\eps)\quad\text{when}\quad p>2^*-1,$$ where $F$ and $\hat F$ are defined as in \eqref{a21} and \eqref{1}. This along
with \eqref{3} and the fact that $F(w_\eps)=S_\eps$ implies
$$F(v_\eps) \to \mathcal{K}\,\,\text{if}\,\ p > 2^*-1.$$

Moreover when $p=2^*-1$,
$$\mathcal{S}\leq S(v_{\eps})\leq 2\hat{F}(v_{\eps}, \Om)=2\hat{F}(w_{\eps},\Om_{\eps})=2S_{\eps}\To\mathcal{S}.$$ 
Hence 
 $$S(v_{\eps})\to\mathcal{S} \quad\text{if}\ p=2^*-1.$$ This completes the proof.
 \end{proof}

{\bf Proof of Theorem \ref{main1}:}  Let $v_\eps$ and $\la_\eps$ be as in Proposition \ref{p:main1}. Define, $ u_\eps=\la_\eps^{\frac{1}{p-1}}v_\eps. $ Then it is easy to see that
$u_\eps$ satisfies $$(-\De)^s u_\eps=u_{\eps}^p -\eps \la_\eps^{\frac{-(q-1)}{p-1}}u_\eps \,\,\text{in}\,\ \Om. $$
Using the bounds on $\la_\eps$ from Proposition \ref{p:main1}, we can conclude that there exist solutions $u_n$  of problem \eqref{eq:a3'} along a sequence $\{\eps_n\}_{n \geq 1}$ of values $\eps$ which tends to $0$ 
as $n \to \infty$. Set $\la_n:=\la_{\eps_n}^{\frac{-1}{p-1}}. $
Thus, from Proposition \ref{p:main1} it follows 
$$F(\la_n u_n)\to \mathcal{K} \quad\text{and}\quad \displaystyle\int_{\Om}u_n^{p+1}\to 0 \quad\text{when}\quad p>2^*-1$$ and $$S(\la_nu_n)\to \mathcal{S} \quad\text{and}\quad A< \displaystyle\int_{\Om}u_n^{p+1}<B \quad\text{when}\quad p=2^*-1,$$ for some $A,B>0$. Since $S(\la_nu_n)=S(u_n)$,  theorem follows. 
\hfil{$\square$}

\section{\bf The case $p=2^*-1$ and the proof of Theorem \ref{main2}}
\begin{lemma}\lab{l:5-1}
Let $u_{\eps}$ be as in Theorem \ref{main2}. Then $\|u_{\eps}\|_{\infty}\rightarrow \infty$ as $\eps \rightarrow 0$. 
\end{lemma}
\begin{proof}
Note that as $u_{\eps}\in C(\bar\Om)$  (see \cite[Theorem 1.2]{BM}), for each fixed $\eps>0$, we have $\|u_{\eps}\|_{\infty}<\infty$.  Furthermore, since $u_{\eps}$ is as in Theorem \ref{main2}, we have
\be \label{zig1} \displaystyle\int_{\Om} u_{\eps}^{2^{\star}} dx=c,\ee where $c\in(A, B)$.
Suppose, $\|u_{\eps}\|_{\infty}$ is uniformly bounded. Therefore, by the Schauder estimate (see \cite{RS2}, \cite{RS1}), $u_{\eps}\rightarrow u$ in $C^{2s-\de}_{loc}(\Om)\cap C^{s-\de}(\Rn)$, for any $\de>0$. By the definition of weak solution, we have
\be\lab{sep-8-1}
\int_{\Rn}\int_{\Rn}\f{(u_{\eps}(x)-u_{\eps}(y))(\va(x)-\va(y))}{|x-y|^{N+2s}}dxdy=\int_{\Om}(u_{\eps}^{2^*-1}-\eps u^{q}_{\eps})\va dx \quad\forall\ \va\in C^{\infty}_0(\Om).
\ee
Moreover, as $||u_{\eps}||_{{C^s}(\Rn)}$ is uniformly bounded (see \cite[Proposition 1.1]{RS1}), we get 
$$\f{(u_{\eps}(x)-u_{\eps}(y))(\va(x)-\va(y))}{|x-y|^{N+2s}}\leq C\f{|x-y|^{s}|\na\va|_{L^{\infty}}|x-y|}{|x-y|^{N+2s}}\leq C\f{1}{|x-y|^{N-1+s}}.$$
Therefore using the dominated convergence theorem, we can pass to the limit in \eqref{sep-8-1} and get, 
\begin{equation}
  \label{b1}
  \left\{
    \begin{aligned}
     (-\De)^s u&= u^{2^*-1}\; &&\text{in } \Om,\\
 u &\geq 0 &&\text{in }  \Om,  \\
u &= 0 &&\text{in } \Rn\setminus\Om,
    \end{aligned}
  \right.
\end{equation}
 where \be \label{zig11} \displaystyle A<\int_{\Om} u^{2^{\star}} dx<B.\ee
As $A>0$, the above expression implies  $u$ is a nontrivial solution in a bounded star-shaped domain. Since, $u\in C(\Rn)$ and $u=0$ in $\Rn\setminus\Om$, clearly $u$ is a bounded solution.  By the maximum principle (\cite[Proposition 2.17]{S}), we also have $u>0$ in $\Om$. This gives 
a contradiction due to the Pohozaev identity \cite[Corollary 1.3]{RS}. Hence the lemma follows.
\end{proof}

\noi Let $x_{\eps}$ be a local maximum point of $u_{\eps}$ and $\ga_{\eps}\in\R^+$ such that \be\lab{ga-eps}
u_{\eps}(x_{\eps})=\|u_{\eps}\|_{\infty}= \gamma_{\eps}^{-\frac{N-2s}{2}}.
\ee
Then  $\gamma_{\eps}\rightarrow 0$ as $\eps\rightarrow 0$.

\begin{lemma}\label{intblow-up} (Blow-up at an interior point)
Let $x_0:=\displaystyle{\lim_{\eps \to 0} }x_{\eps}$, then\  $x_0$ is an interior point of $\Om$.
\end{lemma}
\begin{proof}
Let $\la_1$ be the first eigenvalue of $(-\De)^s$ in $\Om$ and $\va_1$ be a corresponding eigenfunction (see \cite{SerVal2}), that is, $\va_1$ satisfies
\bea
(-\De)^s \va_1&=&\la_1\va \quad\text{in}\quad\Om,\no\\
\va_1&=&0 \quad\text{in}\quad\Rn\setminus\Om.\no
\eea
Moreover, as $u_{\eps}$ is a classical solution (see \cite[Proposition 3.1]{BM}) 
\bea
\la_1\Iom \va_1 u_{\eps}dx=\Iom(-\De)^s\va_1 u_{\eps}dx=\int_{\Rn} (-\De)^s\va_1 u_{\eps} dx &=&\int_{\Rn} \va_1(-\De)^s u_{\eps} dx \no\\
&=&\Iom \va_1 u_{\eps}^{2^*-1}dx-\Iom\va_1 u_{\eps}^qdx\no\\
&\leq& \Iom \va_1 u_{\eps}^{2^*-1}dx\no\\
&\leq&\bigg(\Iom u_{\eps}^{2^*}dx\bigg)^\f{2^*-1}{2^*}\bigg(\Iom \va_1^{2^*}dx\bigg)^\f{1}{2^*}\no\\
&\leq& B^\f{2^*-1}{2^*}\bigg(\Iom \va_1^{2^*}dx\bigg)^\f{1}{2^*}\leq C',
\eea
for some constant $C'$. Hence $\displaystyle\Iom \va_1 u_{\eps}dx\leq\f{C'}{\la_1}$. Since, $\va_1\geq C$ on $\Om'\subset\subset\Om$, we obtain 
\be\lab{Sep-5-1}
\int_{\Om'}u_{\eps}\leq C(\Om'), 
\ee
for any $\Om'\subset\subset\Om$. \\
Define $$O(\de):=\{z\in\Om:\text{dist}(z,\pa\Om)<\de\}$$ and
$$I(\de):=\{z\in\Om:\text{dist}(z,\pa\Om)>\de\}.$$

{\bf Claim:} There exists $C>0$ such that 
$$\sup_{O(\de)}u_{\eps}(x)\leq C \quad\forall\ \eps>0.$$

 If $\Om$ is strictly convex, the moving plane argument , which is given in the proof of \cite[Theorem 3.1]{FW} (also see \cite[Corollary 1.2]{JW}) yields the fact that each solution $u_{\eps}$ increases along an arbitrary straight line toward inside of $\Om$ emanating from a point on $\pa\Om$. (see for instance \cite[Lemma 3.1]{KCL}).  Hence following an argument as in \cite{Z},  we can find $\ga, \de>0$ such that for any $x\in O(\de)$, there exists a measurable set $\Ga_x$ with (i) $\text{meas}(\Ga_x)\geq\ga$, (ii) $\Ga_x\subset I(\f{\de}{2})$, and (iii) $u_{\eps}(y)\geq u_{\eps}(x)$ for any $y\in\Ga_x$. In particular, $\Ga_x$ can be taken as a cone with vertex at $x$. Let $\Om'=I(\f{\de}{2})$. Then for any 
$x\in O(\de)$, we have
$$u_{\eps}(x)\leq\f{1}{\text{meas}(\Ga_x)}\int_{\Ga_x}u_{\eps}(y)dy\leq\ga^{-1}\int_{\Om'}u_{\eps}\leq C(\Om').$$
This proves the claim when $\Om$ is strictly convex. The general case can be proved using Kelvin transform in the extended domain (see, for instance, \cite{Z}, \cite {KCL}, \cite{C}).

From Lemma \ref{l:5-1}, we have $u_{\eps}(x_{\eps})\to\infty$ as $\eps\to 0$. On the other hand, the above claim implies $u_{\eps}$ is uniformly bounded near the boundary for all small $\eps>0$. Hence passing to a subsequence, the point $x_{\eps}$ converges to an interior point $x_0\in\Om$.

\end{proof}

 Define
\begin{equation}\label{z-eps}
z_{\eps}(x)=\gamma_{\eps}^{\frac{N-2s}{2}}u_{\eps} (\gamma_{\eps} x+x_{\eps}).
\end{equation}
Then $\|z_{\eps}\|_{\infty}=1$ and satisfies 
\begin{equation}
  \label{1.41}
  \left\{
    \begin{aligned}  (-\De)^s z_{\eps}& =  z_{\eps}^{2^*-1}-  \eps \gamma_{\eps}^{\f{(N+2s)-q(N-2s)}{2}}  z_{\eps}^{q}  &&\text{ in }    \Om_{\eps}, \\
  z_{\eps}&> 0 &&\text{ in }  \Om_{\eps},\\ z_{\eps}& =0 &&\text{ in }    \Rn\setminus \Om_{\eps},
 \end{aligned}
  \right.
\end{equation}
where $\Om_{\eps}= \f{\Om-x_{\eps}}{\gamma_{\eps}}$.

\begin{lemma}\label{l:Z}
Suppose $z_{\eps}$ is as in \eqref{z-eps}. Then
\begin{itemize}
\item[(i)] $\lim_{\eps\to 0}\eps \gamma_{\eps}^{\f{(N+2s)-q(N-2s)}{2}} =0$\\
\item[(ii)] There exists $Z\in D^{s,2}(\Rn)$ such that  $z_{\eps} \rightarrow Z$ in $C^{2s-\de}_{loc}(\R^N)$ as $\eps \rightarrow 0$, for any $\de>0$.\\
\item[(iii)] $Z$ satisfies Eq. \eqref{ent} and $Z(x)=\bigg[1+\f{|x|^2}{\mu_{N,s}}\bigg]^{-\f{N-2s}{2}}$, 
where $\mu_{N,s}=c_{N,s}^\f{4}{N-2s}$.

\end{itemize}
\end{lemma}
\begin{proof} Using Lemma \ref{intblow-up}, we obtain $\Om_{\eps}\mapsto \R^N$ as $\eps \to 0.$
We know $z_{\eps}$ satisfies Eq.\eqref{1.41}. Note that, $\max_{\Om}u_{\eps}(x)=u_{\eps}(x_\eps)$ implies $z_{\eps}$ attains maximum at $0$ and $z_{\eps}(0)=1$. Therefore,  applying the definition of fractional Laplace operator, it is easy to see that $(-\De)^s z_{\eps}(0)\geq 0$. Thus from \eqref{1.41}, we have $1-  \eps \gamma_{\eps}^{\f{(N+2s)-q(N-2s)}{2}}\geq 0$. This in turn implies, $\lim_{\eps\to 0}\eps \gamma_{\eps}^{\f{(N+2s)-q(N-2s)}{2}}\in[0,1]$. Consequently, using Schauder estimate \cite{RS1}, $z_{\eps}\to Z$ in $C^{2s-\de}_{loc}(\Rn)$, for some $\de>0$. 
Let $\phi\in C^{\infty}_0(\Rn)$. Thus, $\phi\in C^{\infty}_0(\Omega_{\eps})$ for $\eps$ small. Taking $\phi$ as the test function, from Eq.\eqref{1.41} we have 
\bea
\displaystyle\int_{\Rn}\int_{\Rn}\f{(z_{\eps}(x)-z_{\eps}(y))(\phi(x)-\phi(y))}{|x-y|^{N+2s}}dxdy&=&\int_{\Om_{\eps}}z_{\eps}^{2^{*}-1}\phi dx\no\\
&-&\eps\ga_{\eps}^{\f{(N+2s)-q(N-2s)}{2}}\int_{\Om_{\eps}}z_{\eps}^q\phi dx.
\eea
As $||z_{\eps}||_{L^{\infty}}=1$ and $\phi$ has compact support, using dominated convergence theorem as in the proof of Lemma \ref{l:5-1}, we can pass  to the limit $\eps\to 0$ in the above integral identity and obtain 
\begin{equation}
  \label{ent-1'}
\left\{\begin{aligned}
      (-\Delta)^s Z&=Z^{2^*-1}-cZ^{q} \quad\text{in }\quad \R^N,\\ 
      0&<Z\leq 1,   \quad\text{in }\quad \R^N, \quad Z(0)=1,\\ 
       \end{aligned}
  \right.
\end{equation}
where $c=\lim_{\eps\to 0}\eps \gamma_{\eps}^{\f{(N+2s)-q(N-2s)}{2}}$.
 Since $z_{\eps}\in H^s(\Om_{\eps})$ and $z_{\eps}=0$ in $\Rn\setminus\Om_{\eps}$, multiplying \eqref{1.41} by $z_{\eps}$ and integrating over $\Rn$, we have
$$||z_{\eps}||_{D^{s, 2}(\Rn)}^2=\int_{\Rn}z_{\eps}^{2^*}dx-\eps \gamma_{\eps}^{\f{(N+2s)-q(N-2s)}{2}} \int_{\Rn}z_{\eps}^{q+1}dx\leq \int_{\Om_{\eps}}z_{\eps}^{2^*}dx<B.$$
Therefore, up to a subsequence $z_{\eps}\deb\tilde Z$ in $D^{s,2}(\Rn)$. By the uniqueness of limit, $Z=\tilde Z$. Thus $Z\in D^{s,2}(\Rn)$. Consequently, multiplying \eqref{ent-1'} by $Z$ and integrating over $\Rn$, we get $Z\in L^{q+1}(\Rn)$. Hence, if $c\not=0$, we get a contradiction by Pohozaev identity (see \cite[Theorem 1.4]{BM}). This implies $c=0$ and $Z$ satisfies
\eqref{ent}.  As a consequence, $Z$ must be of the form $\xi^{-\f{N-2s}{2}}U(\f{x}{\xi})$, for some $\xi>0$, where $U$ is as in \eqref{ent-U}. As, $\max_{\Om_{\eps}} z_{\eps}=z_{\eps}(0)=1$, we get $Z(0)=1$ and $0\leq Z\leq 1$.  Using this fact, it is easy to see that $\xi=c_{N,s}^\f{2}{N-2s}$, where $c_{N,s}$ is as defined in \eqref{c-ns}. From this, a computation yields $Z(x)=\bigg[1+\f{|x|^2}{\mu_{N,s}}\bigg]^{-\f{N-2s}{2}}$, 
where $\mu_{N,s}=c_{N,s}^\f{4}{N-2s}$.

\end{proof}

\vspace{2mm}
\noi Now we show that there exists $C>0$ independent of $\eps>0$ such that \begin{equation} \label{unif}z_{\eps} (x)\leq C Z(x) \text{ for all }x\in \Om_{\eps}.\end{equation}
The local behavior of $z_{\eps}$ is known. Next, we need to check the behavior of $z_{\eps}$ near $\infty.$ For this, define the Kelvin transform of  $z_{\eps}$ as \be \hat{z}_{\eps}(x)=|x|^{-(N-2s)} z_{\eps}\bigg(\frac{x}{|x|^2}\bigg) \text{ in } \Om_{\eps}\setminus \{0\}.\ee
From \eqref{1.41}, it follows that $\hat{z}_{\eps}$ satisfies 
\begin{equation}
  \label{1.34}
  \left\{
    \begin{aligned}  (-\De)^s\hat{z}_{\eps}& = \hat{z}_{\eps}^{2^*-1}-\eps \gamma_{\eps}^{\f{(N+2s)-q(N-2s)}{2}}|x|^{q(N-2s)-(N+2s)}   \hat{z}_{\eps} ^{q}\quad  \text{in}\quad  \Om^{\star}_{\eps}\\
    \hat{z}_{\eps}&=0\quad \text{in}\quad  \Rn\setminus \Om^{\star}_{\eps}.
\end{aligned}
  \right.
\end{equation}
where $\Om^{\star}_{\eps}$ is the image $\Om_{\eps}$ under the Kelvin transform. Hence the behavior of $z_{\eps}$ near $\infty$ amounts to study the behavior of  $\hat{z}_{\eps}$ near $0.$

\begin{lemma}\label{lb}
There exist $R>0$ and $C>0$ independent of $\eps>0$ such that any solution of \eqref{1.34} satisfy
\be\label{zig2} \|\hat{z}_{\eps}\|_{L^{\infty}(B_{r})}\leq C \bigg(\int_{B_{R}}  \hat{z}_{\eps}^{2^{\star}} dx\bigg)^{\frac{1}{2^{\star}}}.\ee
\end{lemma}
\begin{proof} The proof  follows along the same line of arguments as in \cite[Theorem 1.1]{BM} (see also \cite{TX}) with a suitable modification and we skip the proof.  \end{proof}

For \eqref{unif}, note that $\|z_{\eps}\|_{\infty}=1$ and this implies that $z_{\eps} \leq C Z(x) $ locally.
From \eqref{nov-1-1} and \eqref{z-eps},  it follows
\be\no  A< \int_{\Om_{\eps}}z_{\eps}^{2^*} dx<B. \ee But this implies that
\be\no \int_{B_{R}\cap\Om_{\eps}^*} \hat{z}_{\eps}^{2^{\star}} dx\leq\int_{\Om_{\eps}^*} \hat{z}_{\eps}^{2^{\star}} dx = \int_{\Om_{\eps}} z_{\eps}^{2^*} dx  <B. \ee
Consequently from Lemma \ref{lb}, we obtain $z_{\eps}(x)\leq\f{C}{|x|^{N-2s}}$ as $|x|\to\infty$.
Moreover, since at infinity $Z$ decays as $|x|^{-(N-2s)}$, we conclude $z_{\eps} \leq C Z(x) $ near infinity. Hence, we have $z_{\eps} \leq C Z(x) $ for all $x\in \Om_{\eps}.$
As a conclusion, from \eqref{z-eps} we obtain that there exists $C>0$ independent of $\eps$ such that \be\label{upps} u_{\eps}(x) \leq C \gamma_{\eps}^{-\frac{N-2s}{2}}Z\bigg(\frac{x-x_{\eps}}{\gamma_{\eps}}\bigg) .\ee

\noi Define $w_{\eps}(x)=\|u_{\eps}\|_{\infty}u_{\eps}(x)= \gamma_{\eps}^{-\frac{N-2s}{2}}u_{\eps}(x).$ Then $w_{\eps}$ satisfies
\begin{equation}
  \label{1.37}
  \left\{
    \begin{aligned}  (-\De)^s w_{\eps}& = \gamma_{\eps}^{-\frac{N-2s}{2}} u_{\eps}^{2^*-1}-\eps   \gamma_{\eps}^{-\frac{N-2s}{2}}   u_{\eps}^q \quad\text{ in  } \quad\Om\\
    w_{\eps}&=0 \quad\text{ in  } \quad \Rn\setminus \Om.
\end{aligned}
  \right.
\end{equation}

\begin{lemma}\label{GFC}   The Green function associated to the fractional Laplacian $(- \De)^s$ satisfy the following inequalities.
\begin{itemize}
\item[(i)]  $G(x, y)\leq \frac{C }{|x-y|^{N-2s}}$ and 
 \item[(ii)]$G(x, y)\leq \frac{C d^s(x)}{|x-y|^{N-s}}.$
\end{itemize} 
where $C>0$ is a constant depending on $\Om$ and $s$ and $N> 2 s.$ 
\end{lemma} 

\begin{proof} This follows from Chen and Song  \cite[Theorem 1.1]{CS}.
\end{proof}

\begin{lemma} \label{p8} Let $w_{\eps}$ be as in \eqref{1.37}. Then for every $r>0$, there exists a constant $C=C(r)>0$ such that $$\|w_{\eps}\|_{L^\infty(\Om\setminus B_r(x_0))}\leq C.$$
\end{lemma}
\begin{proof}
From the Green function representation and Lemma \ref{GFC}  we have
\begin{eqnarray}|  w_{\eps}(x)|&\leq&   \gamma_{\eps}^{-\frac{N-2s}{2}}  \int_{\Om} G(x,y)u_{\eps}^{2^*-1} dy+\eps\gamma_{\eps}^{-\frac{N-2s}{2}}  \int_{\Om} G(x,y)u_{\eps}^{q} dy
\no \\&\leq& C \gamma_{\eps}^{-\frac{N-2s}{2}}   \int_{\Om}|x-y|^{2s-N}u_{\eps}^{2^*-1} dy+C\eps \gamma_{\eps}^{-\frac{N-2s}{2}}   \int_{\Om}|x-y|^{2s-N}u_{\eps}^{q} dy. \no
\end{eqnarray}
Moreover,
\begin{eqnarray*}\gamma_{\eps}^{-\frac{N-2s}{2}}  \int_{\Om}|x-y|^{2s-N}u_{\eps}^{2^*-1} dy&=& \gamma_{\eps}^{-\frac{N-2s}{2}}  \int_{\Om \cap B_{\frac{|x- x_{\eps}|}{2}}(x_{\eps})}|x-y|^{2s-N}u_{\eps}^{2^*-1} dy \\&+& \gamma_{\eps}^{-\frac{N-2s}{2}}  \int_{\Om \setminus \Om \cap B_{\frac{|x- x_{\eps}|}{2}}(x_{\eps})}|x-y|^{2s-N}u_{\eps}^{2^*-1} dy. \no
\end{eqnarray*}
Using \eqref{upps} along with that fact that $Z(x)= |x|^{-(N-2s)}$ at infinity, we have
$$\gamma_{\eps}^{-\frac{N-2s}{2}} |x-y|^{2s-N}u_{\eps}^{2^*-1}(y) \leq \frac{C\ga_{\eps}^{2s}}{|x-y|^{N-2s}|y-x_{\eps}|^{N+2s}} \quad\text{if}\  \ y\in \Om\setminus
B_{\frac{|x- x_{\eps}| }{2}}(x_{\eps})$$ and 
$$\eps \gamma_{\eps}^{-\frac{N-2s}{2}} |x-y|^{2s-N}u_{\eps}^{q}(y) dy  \leq \frac{C \eps \ga_{\eps}^{\frac{(N-2s)(q-1)}{2}}}{|x-y|^{N-2s}|y-x_{\eps}|^{(N-2s)q}} \quad\text{if}\  \ y\in \Om\setminus
B_{\frac{|x- x_{\eps}| }{2}}(x_{\eps}). $$

Hence,
\bea &&\gamma_{\eps}^{-\frac{N-2s}{2}}  \int_{\Om\setminus B_{\frac{|x- x_{\eps}|}{2}}(x_{\eps})}|x-y|^{2s-N}u_{\eps}^{2^*-1}(y)dy\no\\
&\leq& \f{C}{|x-x_{\eps}|^{N+2s}}\displaystyle \int_{\Om\setminus B_{\frac{|x- x_{\eps}|}{2}}(x_{\eps})}\frac{1}{|x-y|^{N-2s}} dy\no\\
& \leq & \frac{C}{|x-x_{\eps}|^{N+2s}} \no \eea and 
\Bea && \eps \gamma_{\eps}^{-\frac{N-2s}{2}}  \int_{\Om \setminus \Om \cap B_{\frac{|x- x_{\eps}|}{2}}(x_{\eps})}|x-y|^{2s-N}u_{\eps}^{q}(y) dy \no\\ &\leq&  \frac{C \eps \ga_{\eps}^{\frac{(N-2s)(q-1)}{2}}}{|x-x_{\eps}|^{(N-2s)q}}\displaystyle \int_{\Om\setminus B_{\frac{|x- x_{\eps}|}{2}}(x_{\eps})} \frac{1}{|x-y|^{N-2s}} dy \\&\leq&\frac{C}{|x-x_{\eps}|^{(N-2s)q}}.\Eea

When $y\in\Om\cap B_{\frac{|x- x_{\eps}|}{2}}(x_{\eps})$, we have
$|x-y|\geq |x- x_{\eps}|-|y- x_{\eps}|\geq \frac{1}{2}|x- x_{\eps}|$. Therefore applying \eqref{upps} we obtain 
\Bea  \gamma_{\eps}^{-\frac{N-2s}{2}}  \int_{\Om\cap B_{\frac{|x- x_{\eps}|}{2}}(x_{\eps})} |x-y|^{2s-N} u_{\eps}^{2^*-1}(y)dy
&\leq& \frac{C\gamma_{\eps}^{-\frac{N-2s}{2}} }{|x- x_{\eps}|^{N-2s}}  \int_{\Om\cap B_{\frac{|x- x_{\eps}|}{2}}(x_{\eps})} u_{\eps}^{2^*-1}(y)dy \no \\&\leq& \frac{C\ga_{\eps}^{-N}}{|x- x_{\eps}|^{N-2s}} \int_{\Rn }Z^{2^*-1}\big(\f{y-x_{\eps}}{\ga_{\eps}}\big)dy\no \\&\leq& \frac{C}{|x- x_{\eps}|^{N-2s}}\int_{\Rn}Z^{2^*-1}(x)dx\\&\leq&  \frac{C}{|x- x_{\eps}|^{N-2s}}. \Eea 
Similarly applying Lemma \ref{l:Z}, we obtain
\Bea  \eps \gamma_{\eps}^{-\frac{N-2s}{2}}  \int_{\Om\cap B_{\frac{|x- x_{\eps}|}{2}}(x_{\eps})} |x-y|^{2s-N} u_{\eps}^{q}(y)dy &\leq& \frac{C\eps \gamma_{\eps}^{-\frac{N-2s}{2}} }{|x- x_{\eps}|^{N-2s}}  \int_{\Om\cap B_{\frac{|x- x_{\eps}|}{2}}(x_{\eps})} u_{\eps}^{q}(y)dy \no \\ 
&\leq& \frac{C \eps \gamma_{\eps}^{ N -\frac{N-2s}{2}(q+1)} }{|x- x_{\eps}|^{N-2s}} \int_{\Rn } Z^{q}(y)dy\no \\
&\leq& \frac{C}{|x- x_{\eps}|^{N-2s}}.\Eea
where $C>0$ is a uniform constant.
Hence for any small $r>0$ fixed,  $\Om\setminus B_r(x_0)\subseteq  \Om \setminus \{x_{\eps}\}$, for $\eps>0$ small enough and therefore,
we have $\|w_{\eps}\|_{L^\infty(\Om\setminus B_r(x_0))}\leq C.$
\end{proof}

Note that \eqref{1.37} can be rewritten as  
\begin{equation}
  \label{1.38}
  \left\{
    \begin{aligned}  (-\De)^s w_{\eps}& = \gamma_{\eps}^{2s} w_{\eps}^{2^*-1}-\eps   \gamma_{\eps}^{\frac{N-2s}{2}(q-1)}   w_{\eps}^q \quad\text{ in  } \quad\Om\\
    w_{\eps}&=0 \quad\text{ in  } \quad \Rn\setminus \Om.
\end{aligned}
  \right.
\end{equation}

\begin{lemma}
\label{p7.1} 
\be\label{p7.1'}\lim_{\eps \rightarrow 0}\frac{w_{\eps}(x)}{d(x)^s}=\gamma_{0} \frac{G(x, x_0)}{d(x)^s}~ \text{in}\quad C(\overline{\Om}\setminus B_r( x_0)),\ee
for any $r>0$.  Here, $\ga_0$ is same as in Lemma \ref{A.3}.
\end{lemma}
\begin{proof} Choose $r>0$ such that $\Om'=\Om\setminus \overline{B_{r}}(x_0)$ is connected. Thus by Lemma \ref{p8},
 $ |w_{\eps}|\leq C$ for all $x\in \Om'$. 

Then  for any $r>0$ small and the fact that $\gamma_{\eps}\to 0$ we have
\bea\lab{sep-13-1} \frac{w_{\eps}(x)}{d(x)^s} &=& \frac{\gamma_{\eps}^{-\frac{N-2s}{2}} }{d(x)^s } \int_{\Om} G(x,y)u_{\eps}^{2^*-1} dy- \eps \frac{\gamma_{\eps}^{-\frac{N-2s}{2}} }{d(x)^s }   \int_{\Om} G(x,y)u_{\eps}^{q} dy \no \\&=&  \frac{\gamma_{\eps}^{-\frac{N-2s}{2}} }{d(x)^s }   \int_{B_{r}(x_0)} G(x,y)u_{\eps}^{2^*-1}(y)dy+
\frac{\gamma_{\eps}^{-\frac{N-2s}{2}} }{d(x)^s }   \int_{\Om\setminus B_{r}(x_0)} G(x,y)u_{\eps}^{2^*-1}(y)dy\no\\
&\quad&-\eps \frac{\gamma_{\eps}^{-\frac{N-2s}{2}} }{d(x)^s }   \int_{B_{r}(x_0)} G(x,y)u_{\eps}^{q} dy -\eps \frac{\gamma_{\eps}^{-\frac{N-2s}{2}} }{d(x)^s }   \int_{\Om\setminus B_{r}(x_0)} G(x,y)u_{\eps}^{q} dy.
\eea
Using the second estimate in Lemma \ref{GFC}, \eqref{upps} and the fact that $Z$ decays at infinity of the order $|y|^{-(N-2s)}$, we estimate the 2nd term on RHS as follows 
\Bea  \frac{\gamma_{\eps}^{-\frac{N-2s}{2}} }{d(x)^s }   \int_{\Om \setminus B_{r}(x_0)} G(x,y)u_{\eps}^{2^*-1}(y)dy &\leq &   \gamma_{\eps}^{-\frac{N-2s}{2}}  \int_{\Om \setminus B_{r}(x_0)} \frac{u_{\eps}^{2^*-1}(y)}{|x-y|^{N-s}}  dy\no\\
&=&\gamma_{\eps}^{-N}  \int_{\Om \setminus B_{r}(x_0)}Z^{2^*-1}\bigg(\f{y-x_{\eps}}{\ga_{\eps}}\bigg)|x-y|^{s-N} dy\no\\
&\leq&C\gamma_{\eps}^{-N}  \int_{\Om \setminus B_{r}(x_0)}|\f{y-x_{\eps}}{\ga_{\eps}}|^{-(N+2s)}\frac{1}{|x-y|^{N-s}}dy\no\\
&=&C\ga_{\eps}^{2s}\int_{\Om \setminus B_{r}(x_0)}\f{1}{|y-x_{\eps}|^{(N+2s)} |x-y|^{N-s} }dy\no\\
&=& o_{r, \eps }(1),\Eea
where $o_{r, \eps }(1)$ denote the term going to $0$ as $r\to 0$ or $\eps \to 0.$ Note that we have used the fact that $|x-y|^{s-N}$ is  integrable in $\Om.$ 
Similarly, it can be shown that,
\Bea  \frac{\gamma_{\eps}^{-\frac{N-2s}{2}} }{d(x)^s }   \int_{\Om \setminus B_{r}(x_0)} G(x,y)u_{\eps}^{q}(y)dy &\leq &   \gamma_{\eps}^{-\frac{N-2s}{2}}  \int_{\Om \setminus B_{r}(x_0)} \frac{u_{\eps}^{q}(y)}{|x-y|^{N-s}}  dy\\
&\leq&C\ga_{\eps}^{(\f{N-2s}{2})(q-1)}\int_{\Om \setminus B_{r}(x_0)} \f{1}{|y-x_{\eps}|^{(N-2s)q}  |x-y|^{N-s} } dy\\
&=&o_{r, \eps }(1).\Eea
Furthermore $\frac{G(x,.)}{\de(x)^s}$ is continuous in
$\overline{\Om}\setminus\{x\}$, ( see \cite[Lemma 6.5]{CS} ). Therefore,  from \eqref{sep-13-1} we obtain
\be\lab{sep-13-2}  \frac{w_{\eps}(x)}{d(x)^s} =\gamma_{\eps}^{-\frac{N-2s}{2}}    \frac{G(x, x_{0})}{d(x)^s}  \int_{B_{r}(x_0)} u_{\eps}^{2^*-1}dy+ L+o_{\eps,r}(1),\ee
where $$L=\eps\gamma_{\eps}^{-\frac{N-2s}{2}}    \frac{G(x, x_{0})}{d(x)^s}  \int_{B_{r}(x_0)} u_{\eps}^{q}dy.$$
Doing a straight forward computation using \eqref{upps}, we have
\bea
L&\leq& \eps\ga_{\eps}^\f{(N+2s)-q(N-2s)}{2}\frac{G(x, x_{0})}{d(x)^s}\int_{\f{B_r(x_0)-x_0}{\ga_{\eps}}}Z^q(y)dy \no\\
&\leq& \eps\ga_{\eps}^\f{(N+2s)-q(N-2s)}{2}\frac{G(x, x_{0})}{d(x)^s}\int_{\Rn}Z^q(y)dy\no\eea

Thus, using Lemma \ref{l:Z}, it is not difficult to check that $L=o_{\eps, r}(1)$.
Define
$$ \gamma_{0}=\lim_{r \to 0}\lim_{\eps\to 0}\gamma_{\eps}^{-\frac{N-2s}{2}}   \int_{B_{r}(0)}  u_{\eps}^{2^*-1}dy. $$
Then it follows from \eqref{sep-13-2} that
\be \lim_{\eps \to 0}\frac{w_{\eps}(x)}{d(x)^s} = \gamma_{0} \frac{G(x, x_{0})  }{d(x)^s}.\no\ee
This argument actually goes through for uniform convergence, i.e., we get
\be\lab{10-9-17-2}
\sup_{x\in\Om\setminus B_r(x_0)} \bigg|\f{w_{\eps}(x)}{d^s(x)} - \ga_0\f{G(x, x_0)}{d^s(x)} \bigg| \to 0.
\ee
Furthermore, note that for each fixed $\eps>0$, $\sup_{\Om}|u_{\eps}(x)|<C_{\eps}$.  Thus, from the definition of $w_{\eps}$, we obtain that for each fixed $\eps>0$, RHS of \eqref{1.38} is in $L^{\infty}(\Om)$. Hence for each fixed $\eps>0$, applying \cite[Theorem 1.2]{RS1} we have $\frac{w_{\eps}}{d^s}\in C^{\alpha}(\overline\Om)$, for some $\al\in (0,1)$. On the other hand, from \cite[Lemma 6.5]{CS}, it follows that $\f{G(., x_0)}{d^s}$ is continuous up to $\pa\Om$. Hence, a straight forward elementary analysis yields 
$$ \sup_{x\in\overline{\Om}\setminus B_r(x_0)} \bigg|\f{w_{\eps}(x)}{d^s(x)} - \ga_0\f{G(x, x_0)}{d^s(x)} \bigg|= \sup_{x\in\Om\setminus B_r(x_0)} \bigg|\f{w_{\eps}(x)}{d^s(x)} - \ga_0\f{G(x, x_0)}{d^s(x)} \bigg| \to 0.
$$

Clearly,  $\gamma_{0}$ is positive as
 \begin{eqnarray*}\gamma_{0}&\geq& \lim_{r \to 0}\lim_{\eps\to 0}\gamma_{\eps}^{-\frac{N-2s}{2}}  \|u_{\eps}\|^{-1}_{\infty} \int_{B_{r}(0)} u_{\eps}^{2^*}dy\\&\geq& \lim_{r \to 0}\lim_{\eps\to 0}\int_{B_{r}(0)}  u_{\eps}^{2^*}dy \\&\geq& A . \end{eqnarray*}
This completes the proof.
\end{proof}

\begin{lemma}\label{A.3}
Let  $u_{\eps}$  be as in Theorem \ref{main2} and $\ga_{\eps}$ be as defined in \eqref{ga-eps}.  Define $\ga_0:=\displaystyle{\lim_{r \to 0}\lim_{\eps\to 0}\gamma_{\eps}^{-\frac{N-2s}{2}} \displaystyle\int_{B_{r}(x_0)}  u_{\eps}^{2^*-1}dy. }$ Then
\begin{equation}\lab{ap:ga-0}
\ga_0=\f{\om_{N}c_{N,s}^{2^*}}{2}\frac{\Ga(\frac{N}{2})\Ga(s)}{\Gamma(\frac{N+2s}{2})},
\end{equation}
where  $c_{N,s}$ is as defined in \eqref{c-ns}.
\end{lemma}
\begin{proof}

We define $I_{\eps, r}:=\gamma_{\eps}^{-\frac{N-2s}{2}} \displaystyle\int_{B_{r}(x_0)}u_{\eps}^{2^*-1}dy$.
Using \eqref{z-eps}, we obtain $u_{\eps}(x)=\ga_{\eps}^{-\frac{N-2s}{2}}z_{\eps}\big(\frac{x-x_{\eps}}{\ga_{\eps}}\big)$. Thus
\begin{equation}
I_{\eps,r}=\ga_{\eps}^{-\frac{N-2s}{2}-\f{N+2s}{2}+N}\int_{\frac{B_r(x_0)-x_{\eps}}{\ga_{\eps}}}z_{\eps}^{2^*-1}(x)dx=\int_{\frac{B_r(x_0)-x_{\eps}}{\ga_{\eps}}}z_{\eps}^{2^*-1}(x)dx.
\end{equation}
Note that, $\eps\to 0$ implies $\ga_{\eps}\to 0$. Therefore,
\be\lab{ap:ga-0'}
\ga_0=\lim_{r\to 0}\lim_{\eps\to 0}I_{\eps, r}=\int_{\Rn}Z^{2^*-1}dx,
\ee
 where $Z$ is as in Lemma \ref{l:Z}. Hence, by doing a straight forward computation, we obtain
$$\ga_0=\frac{\om_N c_{N,s}^{2^*}}{2}B\bigg(\frac{N}{2}, s\bigg),$$
where $B(a, b)=\displaystyle\int_{0}^{\infty}t^{a-1}(1+t)^{-a-b}dt$ is the Beta function, $c_{N,s}$ is as defined in \eqref{c-ns} and $\om_N$ is the surface measure of unit ball. Recall that $B(a, b)=\f{\Ga(a)\Ga(b)}{\Ga(a+b)}$.
Thus $B\big(\frac{N}{2}, s\big)=\frac{\Ga(\frac{N}{2})\Ga(s)}{\Gamma(\frac{N+2s}{2})}$ and the lemma follows.
\end{proof}

\vspace{3mm}

\begin{proof}[\bf {Proof of Theorem \ref{main2}}] Applying \eqref{poho} to $u_{\eps}$ yields
$$\Ga(1+s)^2\displaystyle\int_{\pa\Om}\left(\f{u_{\eps}(x)}{d^s(x)}\right)^2 \langle x-x_0,\,\nu\rangle dS= 2 \eps  \bigg( \frac{N-2s}{2}- \frac{N}{q+1}\bigg) \int_{\Om}   u_{\eps}^{q+1} dx .
$$
Using $w_{\eps}=||u_{\eps}||_{\infty}u_{\eps}$ in the above expression, we have
\be\lab{sep-7-1}  \Ga(1+s)^2\displaystyle\int_{\pa\Om}\left(\f{w_{\eps}(x)}{d^s(x)}\right)^2\langle x-x_0,\,\nu\rangle dS= 2\eps\bigg( \frac{N-2s}{2}- \frac{N}{q+1}\bigg)  \|u_{\eps}\|_{\infty}^{2}  \int_{\Om}  u_{\eps}^{q+1} dx.\ee
Thanks to Lemma \ref{p7.1}, applying dominated convergence theorem, we have
\be\lab{oct-5-1}
\lim_{\eps\to 0}\Ga(1+s)^2\displaystyle\int_{\pa\Om}\left(\f{w_{\eps}(x)}{d^s(x)}\right)^2\langle x-x_0,\,\nu\rangle dS= \ga_0^2 \Ga(1+s)^2\displaystyle\int_{\pa\Om}\left(\f{G(x, x_0)}{d^s(x)}\right)^2\langle x-x_0,\,\nu\rangle dS.
\ee 
Moreover, using the relations \eqref{z-eps} and \eqref{ga-eps}, the RHS of \eqref{sep-7-1}  reduces to
\bea\lab{sep-7-2}
\text{RHS of \eqref{sep-7-1}} &=&
 2\eps\bigg( \frac{N-2s}{2}- \frac{N}{q+1}\bigg)  \|u_{\eps}\|_{\infty}^{2}\ga_{\eps}^{\f{(N+2s)-q(N-2s)}{2}}\int_{\Om_{\eps}} z_{\eps}^{q+1} dx \no\\
&=&2\eps\bigg( \frac{N-2s}{2}- \frac{N}{q+1}\bigg) \|u_{\eps}\|_{\infty}^{\f{q(N-2s)+N-6s}{N-2s}} \int_{\Om_{\eps}} z_{\eps}^{q+1} dx .\eea
Since $z_{\eps}\to Z$ a,e and $z_{\eps}\leq CZ$, by the dominated convergence theorem it follows $\displaystyle\int_{\Om_{\eps}}  z_{\eps}^{q+1} dx\to \int_{\Rn}  Z^{q+1} dx$.
We substitute back \eqref{sep-7-2} into \eqref{sep-7-1} and  take
the limit $\eps\to 0$. 
Therefore,  using \eqref{oct-5-1} we obtain 
\be\label{6-i} \lim_{\eps \rightarrow 0 }\eps\|u_{\eps}\|_{\infty}^{\frac{q(N-2s)+N-6s}{N-2s}} = \frac{\ga_0^2\Ga(1+s)^2\displaystyle\int_{\pa\Om}\left(\f{G(x, x_0)}{d^s(x)}\right)^2\langle x-x_0,\,\nu\rangle dS}{2\bigg( \frac{N-2s}{2}- \frac{N}{q+1}\bigg)  \displaystyle\int_{\R^N}Z ^{q+1}dx }.\ee
From Lemma \ref{l:Z}, we know 
 $Z(x)=\big(1+\f{|x|^2}{\mu_{N,s}}\big)^{-\big(\f{N-2s}{2}\big)}$, where $\mu_{N,s}=c_{N,s}^\f{4}{N-2s}$.  Thus, a straight forward calculation yields
$$\displaystyle\int_{\R^N}Z ^{q+1}dx=\f{c_{N,s}^{2^*}\om_N}{2}B\left(\f{N}{2}, \ \big(\f{N-2s}{2}\big)q-s\right).$$ From Lemma \ref{A.3}, it is known that $\ga_0=\f{\om_{N}c_{N,s}^{2^*}}{2}\frac{\Ga(\frac{N}{2})\Ga(s)}{\Gamma(\frac{N+2s}{2})}$. Substituting the value of $\ga_0$ and $\displaystyle\int_{\R^N}Z ^{q+1}dx$ in \eqref{6-i}  we have,
\bea
\lim_{\eps \rightarrow 0 }\eps\|u_{\eps}\|_{\infty}^{\frac{q(N-2s)+N-6s}{N-2s}}=\f{\om_Nc_{N,s}^{2^*}}{2}\f{(q+1)R_{N, s, x_0}}{q(N-2s)-(N+2s)} s^2 \Gamma(s)^2B\bigg(\f{N}{2},\ s\bigg)^2\times\no\\
B\bigg(\f{N}{2},\ \big(\f{N-2s}{2}\big)q-s\bigg)^{-1} \no \eea
\end{proof}

\section{\bf Uniqueness result for $p=2^*-1$}

\noi {\bf Proof of Theorem \ref{t:unique}:} We break the proof into few steps.\\

\noi {\bf Step 1}: Let $u_{\eps}$ and $v_{\eps}$ be two solutions of \eqref{eq:a3'} with $$\max_{\Om}u_{\eps}=\max_{\Om}v_{\eps}.$$ Let 
$\ga_{\eps}$ be as in \eqref{ga-eps}. Then by the assumptions of the theorem, we have 
$$\ga_{\eps}=||u_{\eps}||_{L^{\infty}}^{-\f{2}{N-2s}}=u_{\eps}(0)^{-\f{2}{N-2s}}=||v_{\eps}||_{L^{\infty}(\Om)}^{-\f{2}{N-2s}}=v_{\eps}(0)^{-\f{2}{N-2s}}.$$
Note that, by Lemma \ref{l:5-1}, we have $\ga_{\eps}\to 0$ as $\eps\to 0$.
Define, $$\theta_{\eps}(x)=u_{\eps}(\ga_{\eps}x)-v_{\eps}(\ga_{\eps}x),\quad x\in\Om_{\eps}=\f{\Om}{\ga_{\eps}},$$
and $$\psi_{\eps}(x)=\f{\theta_{\eps}(x)}{||\theta_{\eps}||_{L^{\infty}(\Om_{\eps})}}=\f{\theta_{\eps}(x)}{||u_{\eps}-v_{\eps}||_{L^{\infty}(\Om)}}.$$

Therefore, 
$$(-\De)^s\psi_{\eps}=\f{\ga_{\eps}^{2s}}{||u_{\eps}-v_{\eps}||_{L^{\infty}(\Om)}}[\big(u_{\eps}^p(\ga_{\eps}x)-v_{\eps}^p(\ga_{\eps}x)\big)-\eps\big(u_{\eps}^q(\ga_{\eps}x)-v_{\eps}^q(\ga_{\eps}x)\big)].$$ 
It is easy to see that,
$$u_{\eps}^p(\ga_{\eps}x)-v_{\eps}^p(\ga_{\eps}x)=p\int_0^1 \big(tu_{\eps}(\ga_{\eps}x)+(1-t)v_{\eps}(\ga_{\eps}x)\big)^{p-1}\theta_{\eps}(x) dt.$$
Using the fact that $p=2^*-1=\f{N+2s}{N-2s}$ and $\ga_{\eps}^{2s}=||u_{\eps}||_{L^{\infty}(\Om)}^{-(p-1)}=||v_{\eps}||_{L^{\infty}(\Om)}^{-(p-1)}$, a straight forward computation yields
\begin{equation}\lab{nov-2-1}
\left\{\begin{aligned}
      (-\De)^s \psi_{\eps} &=\big(c^1_{\eps}(x)-\eps c^2_{\eps}(x)\big)\psi_{\eps} \quad\text{in }\quad \Om_{\eps}, \\
      \psi_{\eps} &=0 \quad\text{in}\quad \Rn\setminus\Om_{\eps},
          \end{aligned}
  \right.
\end{equation}
where
\be\lab{nov-2-2}
c^1_{\eps}(x)=p\int_0^1\bigg[t\f{u_{\eps}(\ga_{\eps}x)}{||u_{\eps}||_{L^{\infty}(\Om)}}+(1-t)\f{v_{\eps}(\ga_{\eps}x)}{||v_{\eps}||_{L^{\infty}(\Om)}}\bigg]^{p-1}dt,\ee
\be\lab{nov-2-3}
c^2_{\eps}(x)=q\ga_{\eps}^{\f{(N+2s)-q(N-2s)}{2}}\int_0^1\bigg[t\f{u_{\eps}(\ga_{\eps}x)}{||u_{\eps}||_{L^{\infty}(\Om)}}+(1-t)\f{v_{\eps}(\ga_{\eps}x)}{||v_{\eps}||_{L^{\infty}(\Om)}}\bigg]^{q-1}dt.
\ee
Here we observe that, $\f{u_{\eps}(\ga_{\eps}x)}{||u_{\eps}||_{L^{\infty}(\Om)}}=z_{\eps}(x)$, where $z_{\eps}$ is as defined in \eqref{z-eps} (since here $x_{\eps}=0$). Consequently, using Lemma \ref{l:Z} and \eqref{unif}, we obtain
\be\lab{nov-2-12}\f{u_{\eps}(\ga_{\eps}x)}{||u_{\eps}||_{L^{\infty}(\Om)}}\to Z \quad\text{in}\ C^{s}_{loc}(\Rn) \quad\text{and}\quad \f{u_{\eps}(\ga_{\eps}x)}{||u_{\eps}||_{L^{\infty}(\Om)}}\leq \f{C}{\big(1+\f{|x|^2}{\mu_{N,s}}\big)^\f{N-2s}{2}},\ee
where $Z$ is the solution of \eqref{ent} with $Z(0)=1$ and $0< Z\leq 1$.
Hence $Z(x)=(1+\f{|x|^2}{\mu_{N,s}})^{-\f{N-2s}{2}}$, where $\mu_{N,s}=c_{N,s}^\f{4}{N-2s}$, (see Lemma \ref{l:Z}). As a consequence, thanks to Lemma \ref{l:Z}(i), from \eqref{nov-2-2} and \eqref{nov-2-3} we have
\be\lab{nov-2-4}
c^1_{\eps}(x)\to \bigg(\frac{N+2s}{N-2s}\bigg) \f{1}{\big(1+\f{|x|^2}{\mu_{N,s}}\big)^{2s}} \quad\text{and}\quad \eps c^2_{\eps}(x)\to 0,
\ee
uniformly on compact subsets of $\Rn$. Applying Schauder estimates \cite{RS1} to the equation \eqref{nov-2-1}, it follows there exists $\psi\in C^s(\Rn)$ such that $\psi_{\eps}\to \psi$ in $C^s_{loc}(\Rn)$. Since, from Remark \ref{r:nov-1} we have $\psi_{\eps}$ is radially symmetric, 
we obtain $\psi$ is radially symmetric too. Passing to the limit in \eqref{nov-2-1} (as in Lemma \ref{l:Z}) yields
\begin{equation}\lab{nov-2-5}
\left\{\begin{aligned}
      &(-\De)^s \psi =\bigg(\frac{N+2s}{N-2s}\bigg)\f{\psi}{\big(1+\f{|x|^2}{\mu_{N,s}}\big)^{2s}} \quad\text{in }\quad \Rn, \\
      &||\psi||_{L^{\infty}(\Rn)}\leq 1.
          \end{aligned}
  \right.
\end{equation}

\vspace{2mm}

\noi {\bf Step 2}: In this step, we will prove that $\psi\in D^{s,2}(\Rn)$.\\
Since $\psi_{\eps}\in H^s(\Om_{\eps})$, $\psi_{\eps}=0$ in $\Rn\setminus\Om_{\eps}$ and $u_{\eps}, v_{\eps}=0$ in $\Rn\setminus\Om$, taking $\psi_{\eps}$ as a test function in \eqref{nov-2-1}, we have
\be\lab{nov-2-6}
||\psi_{\eps}||^2_{D^{s,2}(\Rn)}= \int_{\Rn}c^1_{\eps}(x)\psi_{\eps}^2dx-\eps\int_{\Rn}c^2_{\eps}(x)\psi_{\eps}^2dx\leq \int_{\Om_{\eps}}c^1_{\eps}(x)\psi_{\eps}^2dx.
\ee
Thus applying the Sobolev inequality, we have
\be\lab{nov-2-7}
\mathcal{S}\bigg(\int_{\Om_{\eps}}|\psi_{\eps}|^{2^*}dx\bigg)^\f{2}{2^*}\leq \int_{\Om_{\eps}}c^1_{\eps}(x)\psi_{\eps}^2dx.
\ee
Let us fix $\de>0$, will be chosen later. Since $||\psi_{\eps}||_{L^{\infty}(\Om_{\eps})}=1$,  H\"{o}lder inequality yields
\be\lab{nov-2-8}
\int_{\Om_{\eps}}c^1_{\eps}(x)\psi_{\eps}^2dx\leq \int_{\Om_{\eps}}c^1_{\eps}(x)\psi_{\eps}^{2-\de}dx\leq\bigg(\int_{\Om_{\eps}}|\psi_{\eps}|^{2^*}dx\bigg)^\f{2-\de}{2^*}\bigg(\int_{\Om_{\eps}}|c^1_{\eps}|^{\f{2^*}{2^*-2+\de}}dx\bigg)^\f{2^*-2+\de}{2^*}. 
\ee
Combining \eqref{nov-2-7} and \eqref{nov-2-8} we have
\bea\lab{nov-2-9}\int_{\Om_{\eps}}|\psi_{\eps}|^{2^*}dx &\leq&\bigg(\int_{\Om_{\eps}}|c^1_{\eps}|^{\f{2^*}{2^*-2+\de}}dx\bigg)^\f{2^*-2+\de}{\de}\no\\
&\leq& C\bigg(\int_{\Rn}\bigg[\f{1}{\big(1+\f{|x|^2}{\mu_{N,s}}\big)^{2s}}\bigg]^{\f{2^*}{2^*-2+\de}}dx\bigg)^\f{2^*-2+\de}{\de}\no\\
&\leq& C, \eea
for some constant $C>0$, if we choose $\de<\f{4s}{N-2s}$. For this choice of $\de$, substituting back \eqref{nov-2-9} into \eqref{nov-2-8} yields $\displaystyle\int_{\Om_{\eps}}c^1_{\eps}(x)\psi_{\eps}^2dx\leq C$. As a result, from \eqref{nov-2-6} we have $||\psi_{\eps}||_{D^{s,2}(\Rn)}$ is uniformly bounded. Since $\psi_{\eps}\to \psi$ in $C^s_{loc}(\Rn)$,
$$||\psi||_{D^{s,2}(\Rn)}\leq \lim \inf_{\eps\to 0}||\psi_{\eps}||_{D^{s,2}(\Rn)}\leq C,$$ which implies $\psi\in D^{s,2}(\Rn)$.

\vspace{2mm}

\noi{\bf Step 3}: In this step we will establish that 
\be\lab{nov-2-10}
|\psi_{\eps}(x)|\leq \f{C}{|x|^{N-2s}}, \quad x\in\Om_{\eps}\setminus B_r(0),
\ee
for $\eps>0$ small enough and for some constant $C>0$ and $r>0$ independent of $\eps$. \\

To prove this step, define $\hat{\psi}_{\eps}$ as the Kelvin transform of $\psi_{\eps}$, that is,
$$\hat{\psi}_{\eps}(x)=\f{1}{|x|^{N-2s}}\psi_{\eps}(\f{x}{|x|^2}), \quad x\in\Om_{\eps}\setminus\{0\}.$$ Let $\Om^{\star}_{\eps}$ be the image $\Om_{\eps}$ under the Kelvin transform.  Since \\
$(-\De)^s \hat{\psi}_{\eps}(x)= \f{1}{|x|^{N+2s}}(-\De)^s\psi_{\eps}(\f{x}{|x|^2})$, doing a straight forward computation we obtain,
\begin{equation}\lab{nov-2-11}
\left\{\begin{aligned}
      &(-\De)^s \hat\psi_{\eps} =\f{1}{|x|^{4s}}\bigg(c^1_{\eps}\big(\f{x}{|x|^2}\big)-\eps c^2_{\eps}\big(\f{x}{|x|^2}\big)\bigg)\hat\psi_{\eps} \quad\text{in }\quad \Om_{\eps}^*, \\
      &\hat\psi_{\eps}=0 \quad\text{in }\quad \Rn\setminus\Om_{\eps}^*.
          \end{aligned}
  \right.
\end{equation}
We set, $$a_{\eps}(x):= \f{1}{|x|^{4s}}\bigg(c^1_{\eps}\big(\f{x}{|x|^2}\big)-\eps c^2_{\eps}\big(\f{x}{|x|^2}\big)\bigg).$$
Thus, \eqref{nov-2-11} reduces to $$(-\De)^s \hat\psi_{\eps} =a_{\eps}(x)\hat\psi_{\eps} \quad\text{in }\quad \Om_{\eps}^*.$$

\noi{\bf Claim:} For $N>4s$, the function $a_{\eps}\in L^{t}(\Om_{\eps}^*)$, for some $t>\f{N}{2s}$. \\

Assuming the claim, let us first complete the proof of step 3. Thanks to the above claim, using Moser iteration technique in the spirit of the proof of \cite[Theorem 1.1]{BM} (see also \cite{TX} and \cite[Lemma B.3]{St}), it can be shown that  
$$\sup_{\Om_{\eps}^*\cap B_1(0)}|\hat\psi_{\eps}|\leq C\bigg(\int_{\Om_{\eps}^*\cap B_2(0)}|\hat\psi_{\eps}|^{2^*}\bigg)^\f{1}{2^*}.$$
Moreover,
$$\int_{\Om_{\eps}^*\cap B_2(0)}|\hat\psi_{\eps}|^{2^*}\leq \int_{\Om_{\eps}^*}|\hat\psi_{\eps}|^{2^*}=\int_{\Om_{\eps}}|\psi_{\eps}|^{2^*}\leq C.$$
The last inequality is due to \eqref{nov-2-9}. Hence $\sup_{\Om_{\eps}^*\cap B_1(0)}|\hat\psi_{\eps}|\leq C$. This in turn implies, 
$$|\psi_{\eps}(x)|\leq \f{C}{|x|^{N-2s}}, \quad x\in\Om_{\eps}\setminus B_r(0),$$
for $\eps>0$ small enough and for some constant $C>0$ and $r>0$.

Now, let us prove the claim.

Using \eqref{nov-2-12}, it is easy to see that $\f{1}{|x|^{4s}}c^1_{\eps}(\f{x}{|x|^2})\leq \f{C}{\big(\mu_{N,s}^{-1}+|x|^{2}\big)^{2s}}$.
Hence for $t>\f{N}{2s}$,
\be\lab{nov-2-13}
\int_{\Om_{\eps}^*}\f{1}{|x|^{4st}}c^1_{\eps}(\f{x}{|x|^2})^t dx\leq C\int_{\Rn}\f{dx}{(\mu_{N,s}^{-1}+|x|^{2})^{2st}}<\infty.
\ee
 On the other hand, $||\f{u_{\eps}(\ga_{\eps}x)}{||u_{\eps}||_{L^{\infty}(\Om)}}||_{L^{\infty}(\Om)}\leq 1$ implies $|\eps c^2_{\eps}|\leq q\eps\ga_{\eps}^{\f{(N+2s)-q(N-2s)}{2}} $. Note that, boundedness of $\Om$ implies there exists $R>0$ such that $\Om\subseteq B_R(0)$. Hence $\Om_{\eps}\subseteq B_\f{R}{\ga_{\eps}}(0)$ and 
$\Om_{\eps}^*\subseteq \Rn\setminus B_\f{\ga_{\eps}}{R}(0)$. 
Therefore,
\bea\lab{nov-2-14}
\int_{\Om_{\eps}^*}\f{1}{|x|^{4st}}c^2_{\eps}(\f{x}{|x|^2})^t dx &\leq& C\bigg[\eps\ga_{\eps}^{\f{(N+2s)-q(N-2s)}{2}}\bigg]^t \int_{\Om_{\eps}^*}\f{dx}{|x|^{4st}}\no\\
&\leq& C\bigg[\eps\ga_{\eps}^{\f{(N+2s)-q(N-2s)}{2}}\bigg]^t \bigg(\f{\ga_{\eps}}{R}\bigg)^{N-4st}
\eea
Since $p=2^*-1$, from Theorem \ref{main2}, it follows that $\eps||u_{\eps}||^\f{q(N-2s)+N-6s}{N-2s}=C'$, that is, $\eps\ga_{\eps}^{-\f{N-6s+q(N-2s)}{2}}=C'$. As a result,
\bea\lab{nov-2-15}
\text{RHS of \eqref{nov-2-14}}\leq C\ga_{\eps}^{t(N-6s)+N}.
\eea
Clearly, $N\geq 6s$ implies $\ga_{\eps}^{t(N-6s)+N}<C$ for some constant $C>0$. If $4s<N<6s$, then choose $t\in(\f{N}{2s}, \f{N}{6s-N})$ to get $t(N-6s)+N\geq 0$. 

Hence, combining \eqref{nov-2-13} and \eqref{nov-2-15} the claim follows.

\vspace{2mm}

{\bf Step 4}: Thanks to \cite[Theorem 1.1]{DPS},  the linear space of solutions to equation \eqref{nov-2-5} can be spanned by the following $(N+1)$ functions:
$$\psi_i(x)=\f{2x_i}{\big(1+\f{|x|^2}{\mu_{N,s}}\big)^\f{N-2s+2}{2}}, \quad i=1,\cdots, N$$and
$$\psi_{N+1}(x)=\f{1-|x|^2}{\big(1+\f{|x|^2}{\mu_{N,s}}\big)^\f{N-2s+2}{2}}.$$
That is, general solution of \eqref{nov-2-5} can be written as
$$\psi(x)=\al \f{1-|x|^2}{\big(1+\f{|x|^2}{\mu_{N,s}}\big)^\f{N-2s+2}{2}}+\sum_{i=1}^N \ba_i \f{2x_i}{\big(1+\f{|x|^2}{\mu_{N,s}}\big)^\f{N-2s+2}{2}},$$ where $\al, \ba_i\in\R$.
Since $\psi$ is a symmetric function, each $\ba_i=0$. 

\vspace{2mm}

{\bf Step 5}: In this step we will prove that  $\al=0$.

 Suppose $\al\not\not=0$. We aim to get a contradiction. For simplicity of the calculation, we can take  $\al=1$ and $\mu_{N,s}=1$, that is, 
\be\lab{nov-3-4}\psi(x)=\f{1-|x|^2}{(1+|x|^2)^\f{N-2s+2}{2}}.\ee

Let $\Om'$ be any neighbourhood of $\pa\Om$, not containing the origin.  

\vspace{2mm}

\begin{center}
{\bf Claim}: $||u_{\eps}||_{L^{\infty}(\Om)}^2\f{(u_{\eps}(x)-v_{\eps}(x))}{||u_{\eps}-v_{\eps}||_{L^{\infty}(\Om)}\de(x)^s}\to -c_0\f{G(x,0)}{\de(x)^s} \quad\text{uniformly in}\  \Om',$
\end{center}
for some constant $c_0>0$.

\vspace{2mm}

 Indeed,
\bea\lab{nov-3-1}
(-\De)^s \bigg(||u_{\eps}||_{L^{\infty}(\Om)}^2\f{(u_{\eps}(x)-v_{\eps}(x))}{||u_{\eps}-v_{\eps}||_{L^{\infty}(\Om)}}\bigg) 
&=&\f{||u_{\eps}||^2_{L^{\infty}(\Om)}}{||u_{\eps}-v_{\eps}||_{L^{\infty}(\Om)}}\bigg[(u_{\eps}^p-v_{\eps}^p)-\eps(u_{\eps}^q-v_{\eps}^q) \bigg]\no\\
&=&\f{||u_{\eps}||^2_{L^{\infty}(\Om)}}{||u_{\eps}-v_{\eps}||_{L^{\infty}(\Om)}}(d^1_{\eps}(x)-\eps d^2_{\eps}(x))(u_{\eps}-v_{\eps})\no\\
&=&: f_{\eps},
\eea 
where $$d^1_{\eps}(x)=p\int_0^1\big(tu_{\eps}(x)+(1-t)v_{\eps}(x)\big)^{p-1}dt$$
and $$d^2_{\eps}(x)=q\int_0^1\big(tu_{\eps}(x)+(1-t)v_{\eps}(x)\big)^{q-1}dt.$$
Note that $$d^1_{\eps}(\ga_{\eps}x)=\ga_{\eps}^{-2s}c^1_{\eps}(x) \quad\text{and}\quad d^2_{\eps}(\ga_{\eps}x)=\ga_{\eps}^{-2s}c^2_{\eps}(x).$$
Therefore, using \eqref{nov-2-12}, we have
\be\lab{nov-3-2}
d^1_{\eps}(x)\leq C \ga_{\eps}^{-2s}\f{1}{(\mu_{N,s}+|\f{x}{\ga_{\eps}}|^2)^{2s}}\leq C\f{\ga_{\eps}^{2s}}{|x|^{4s}}.
\ee
\be\lab{nov-3-3}
d^2_{\eps}(x)\leq C\f{\ga_{\eps}^{-2s}}{(\mu_{N,s}+|\f{x}{\ga_{\eps}}|^2)^{\f{(N-2s)(q-1)}{2}}}\leq C\f{\ga_{\eps}^{q(N-2s)-N}}{|x|^{(N-2s)(q-1)}}.
\ee
{\bf Subclaim 1:} $\lim_{\eps\to 0}f_{\eps}(x)=0 \quad \forall\ x\in\Om'$.\\
As $\ga_{\eps}\to 0$, using \eqref{nov-2-10}, \eqref{nov-3-2} and \eqref{nov-3-3},  for $x\in\Om'$ we obtain

\bea
f_{\eps}(x)&=&||u_{\eps}||_{L^{\infty}(\Om)}^2\psi_{\eps}\bigg(\f{x}{\ga_{\eps}}\bigg)\big(d^1_{\eps}(x)-\eps d^2_{\eps}(x)\big)\no\\
&\leq& C||u_{\eps}||_{L^{\infty}(\Om)}^2\f{1}{|\f{x}{\ga_{\eps}}|^{N-2s}}\big(d^1_{\eps}(x)+\eps d^2_{\eps}(x)\big)\no\\
&\leq&\f{C}{|x|^{N-2s}}\bigg(\f{\ga_{\eps}^{2s}}{|x|^{4s}}+\f{\ga_{\eps}^{q(N-2s)-N}}{|x|^{(N-2s)(q-1)}}\bigg)\no\\
&\to& 0,\no
\eea
since $q>\f{N+2s}{N-2s}$.

\noi{\bf Subclaim 2}: $\displaystyle{\lim_{\eps\to 0}}\int_{\Om}f_{\eps}(x) dx=-c_0$, for some constant $c_0>0$. \\

To see this, 
\bea \displaystyle\int_{\Om}f_{\eps}(x) dx&=& \f{||u_{\eps}||^2_{L^{\infty}(\Om)}}{||u_{\eps}-v_{\eps}||_{L^{\infty}(\Om)}}\int_{\Om}d^1_{\eps}(x)(u_{\eps}-v_{\eps})dx\no\\
&-&\f{||u_{\eps}||^2_{L^{\infty}(\Om)}}{||u_{\eps}-v_{\eps}||_{L^{\infty}(\Om)}}\int_{\Om}\eps d^2_{\eps}(x)(u_{\eps}-v_{\eps})dx\no\\
&=&\int_{\Om_{\eps}}c^1_{\eps}(y)\psi_{\eps}(y) dy- \eps\int_{\Om_{\eps}}c^2_{\eps}(y)\psi_{\eps}(y) dy.\no
\eea
In the last step, we have used the change of variable $x=\ga_{\eps}y$.
Using \eqref{nov-2-4} and \eqref{nov-3-4} via dominated convergence theorem, we obtain
\be\lab{nov-3-5}\lim_{\eps\to 0}\int_{\Om_{\eps}}c^1_{\eps}(y)\psi_{\eps}(y) dy=p\int_{\Rn}\f{1-|x|^2}{(1+|x|^2)^{2s+\f{N-2s+2}{2}}}dx.\ee 
Using change of variable the RHS of the above equality can be computed as follows:
\bea\lab{nov-3-9}
\int_{\Rn}\f{1-|x|^2}{(1+|x|^2)^\f{N+2s+2}{2}}dx&=&\om_N\int_0^1\f{(1-r^2)r^{N-1}}{(1+r^2)^\f{N+2s+2}{2}}dr\no\\
&-&\om_N\int_1^0\f{1-\f{1}{t^2}}{(1+\f{1}{t^2})^\f{N+2s+2}{2}}t^{-2-(N-1)}dt \no\\
&=&-\om_N\int_0^1\f{r^{2s-1}(1-r^2)(1-r^{N-2s})}{(1+r^2)^\f{N+2s+2}{2}}dr
\eea
As $s>0$, $\displaystyle\int_0^1\f{r^{2s-1}(1-r^2)(1-r^{N-2s})}{(1+r^2)^\f{N+2s+2}{2}}dr\leq\int_0^1 r^{2s-1}dr<\infty$. Hence from \eqref{nov-3-5}, we get
\be\lab{nov-3-6}
\lim_{\eps\to 0}\int_{\Om_{\eps}}c^1_{\eps}(y)\psi_{\eps}(y) dy=-c_0,
\ee
for some $c_0>0$. Similarly it can be shown that
$$|\lim_{\eps\to 0}\int_{\Om_{\eps}}c^2_{\eps}(y)\psi_{\eps}(y) dy|<\infty.$$ Therefore, 
\be\lab{nov-3-7}
\lim_{\eps\to 0}\eps\int_{\Om_{\eps}}c^2_{\eps}(y)\psi_{\eps}(y) dy=0.
\ee
Combining \eqref{nov-3-6} and \eqref{nov-3-7}, Subclaim 2 follows.

Now we get back to \eqref{nov-3-1}. Define, $$\phi_{\eps}(x):=
||u_{\eps}||_{L^{\infty}(\Om)}^2\f{(u_{\eps}(x)-v_{\eps}(x))}{||u_{\eps}-v_{\eps}||_{L^{\infty}(\Om)}}.$$ Then $\phi_{\eps}$ satisfies
\begin{equation*}
\left\{\begin{aligned}
      (-\De)^s \phi_{\eps}& =f_{\eps} \quad\text{in }\quad \Om, \\
      \phi_{\eps}&=0 \quad\text{in }\quad \Rn\setminus\Om.\\
          \end{aligned}
  \right.
\end{equation*}
Then for any $r>0$ small and $x\in\Om'$, we have
\bea\lab{nov-3-14}
\f{\phi_{\eps}(x)}{d^s(x)}&=&\int_{\Om}\f{G(x,y)f_{\eps}(y)}{d^s(x)}dy\no\\
&=&\int_{B_r(0)}\f{G(x,y)f_{\eps}(y)}{d^s(x)}dy+\int_{\Om\setminus B_r(0)}\f{G(x,y)f_{\eps}(y)}{d^s(x)}dy.
\eea
Using Lemma \ref{GFC} and  Subclaim 1, we estimate the 2nd term on RHS as follows:
\be
\bigg|\int_{\Om\setminus B_r(0)}\f{G(x,y)f_{\eps}(y)}{d^s(x)}dy\bigg|\leq C\int_{\Om\setminus B_r(0)}\f{|f_{\eps}(y)|}{|x-y|^{N-s}}dy=o_{\eps, r}(1),
\ee
where $o_{r, \eps }(1)$ denote the term going to $0$ as $r\to 0$ and $\eps \to 0.$ Note that we have used the fact that $|x-y|^{s-N}$ is  integrable in $\Om.$ 
Furthermore $\frac{G(x,.)}{d(x)^s}$ is continuous in
$\overline{\Om}\setminus\{x\}$, ( see \cite[Lemma 6.5]{CS}). Therefore from \eqref{nov-3-14}, we obtain  
 $$\lim_{\eps\to 0}\f{\phi_{\eps}(x)}{d^s(x)}=\f{G(x,0)}{d^s(x)}\lim_{r\to 0}\lim_{\eps\to 0}\int_{B_r(0)}f_{\eps}(y)dy.$$
Moreover, by Subclaim 2, 
$$\lim_{r\to 0}\lim_{\eps\to 0}\f{G(x,0)}{d^s(x)}\int_{B_r(0)}f_{\eps}(y)dy=-c_0\f{G(x,0)}{d^s(x)}.$$ Thus, it follows
\be\lab{nov-3-8}
\lim_{\eps\to 0}\f{\phi_{\eps}(x)}{d^s(x)} =-c_0\f{G(x,0)}{d^s(x)}.\ee
This proves the claim.

\vspace{2mm}

In order to complete the proof of Step 5, we apply the Pohozaev identity \eqref{poho-1} to $u_{\eps}$ and $v_{\eps}$.
$$\Ga(1+s)^2\int_{\pa\Om}\bigg(\f{u_{\eps}(x)}{d^s(x)}\bigg)^2(x\cdot\nu)dS=\eps\bigg[(N-2s)-\f{2N}{q+1}\bigg]\int_{\Om}u_{\eps}^{q+1}dx,$$
$$\Ga(1+s)^2\int_{\pa\Om}\bigg(\f{v_{\eps}(x)}{d^s(x)}\bigg)^2(x\cdot\nu)dS=\eps\bigg[(N-2s)-\f{2N}{q+1}\bigg]\int_{\Om}v_{\eps}^{q+1}dx.$$
Subtracting one from the other and multiplying by $\f{||u_{\eps}||^3_{L^{\infty}(\Om)}}{||u_{\eps}-v_{\eps}||_{L^{\infty}(\Om)}}$ in both sides yields,
\bea\lab{nov-4-1}
&\displaystyle\quad\quad\quad\Ga(1+s)^2\int_{\pa\Om}\f{||u_{\eps}||^2_{L^{\infty}(\Om)}(u_{\eps}-v_{\eps})}{||u_{\eps}-v_{\eps}||_{L^{\infty}(\Om)}d^s(x)}\f{(u_{\eps}+v_{\eps})||u_{\eps}||_{L^{\infty}(\Om)}}{d^s(x)}(x\cdot\nu)dS\\
&=\displaystyle\eps\bigg[(N-2s)-\f{2N}{q+1}\bigg](q+1)\int_{\Om}\f{||u_{\eps}||^3_{L^{\infty}(\Om)}}{||u_{\eps}-v_{\eps}||_{L^{\infty}(\Om)}}(u_{\eps}-v_{\eps})\int_0^1(tu_{\eps}+(1-t)v_{\eps})^q dt dx\no.
\eea 
By doing the change of variable $x=\ga_{\eps}y$,  RHS of \eqref{nov-4-1} reduces as
\bea
\text{RHS of \eqref{nov-4-1}}&=&\eps||u_{\eps}||^{q-p+2}\big[q(N-2s)-(N+2s)\big]\no\\
&\times&\int_{\Om_{\eps}}\psi_{\eps}(y)\bigg[\int_0^1\bigg(t\f{u_{\eps}(\ga_{\eps}y)}{||u_{\eps}||_{L^{\infty}(\Om)}}+(1-t)\f{v_{\eps}(\ga_{\eps}y)}{||v_{\eps}||_{L^{\infty}(\Om)}}\bigg)^q dt \bigg]dy.\no
\eea
Note that By Theorem \ref{main2}, $\lim_{\eps\to 0}\eps||u_{\eps}||^{q-p+2}\big[q(N-2s)-(N+2s)\big]=C_1$, for some constant $C_1>0$. Therefore, using  dominated convergence theorem via \eqref{nov-2-12} and \eqref{nov-3-4} , we obtain
 \bea
 \lim_{\eps\to 0}\ \text{RHS of \eqref{nov-4-1}}=C_1\int_{\Rn}\f{1-|x|^2}{(1+|x|^2)^\f{N-2s+2+q(N-2s)}{2}}dx.
 \eea
Applying the change of variable as in \eqref{nov-3-9}, it can be proved  that 
\bea\int_{\Rn}\f{1-|x|^2}{(1+|x|^2)^\f{N-2s+2+q(N-2s)}{2}}dx&=& \om_N\int_0^1\f{r^{N-1}(1-r^2)(1-r^{q(N-2s)-(N+2s)})}{(1+r^2)^{\f{(N-2s)(q+1)}{2}+1}}dr\no\\
&=&C_2,\no
\eea
where $C_2>0$ is a constant. Hence, \bea\lab{nov-3-10}
\lim_{\eps\to 0}\ [\text{RHS of \eqref{nov-4-1}}]>0.\eea
On the other hand, applying \eqref{nsd2} and \eqref{nov-3-8} to LHS via dominated convergence theorem, we get
\bea\lab{nov-3-11}
\lim_{\eps\to 0}\ [\text{LHS of \eqref{nov-4-1}}]&=& 2\Ga(1+s)^2\int_{\pa\Om}-c_0\f{G(x,0)}{d^s(x)} \f{\om_{N}c_{N,s}^{2^*-1}}{2}\frac{\Ga(\frac{N}{2})\Ga(s)}{\Gamma(\frac{N+2s}{2})}\f{G(x, 0)}{d^s(x)}(x\cdot\nu)dS\no\\
&=&  -\f{c_0\om_{N}c_{N,s}^{2^*-1}}{2}\frac{\Ga(\frac{N}{2})\Ga(s)\Ga(1+s)^2}{\Gamma(\frac{N+2s}{2})}\int_{\pa\Om}\bigg(\f{G(x,0)}{d^s(x)}\bigg)^2(x\cdot\nu)dS\no\\
&<&0.
\eea
Combining \eqref{nov-3-10} along with \eqref{nov-3-11} gives the contradiction. Hence $\al=0$ and step 5 follows.
 
 \vspace{2mm}
 
\noi{\bf Step 6}: Step 5 implies that $\psi\equiv 0$. Therefore, by Step 1, $\psi_{\eps}\to 0$ in $K$ for every compact set $K$ in $\Rn$ . Let $y_{\eps}\in\Rn$ such that
$$\psi_{\eps}(y_{\eps})=||\psi_{\eps}||_{L^{\infty}(\Om_{\eps})}.$$
Since by definition of $\psi_{\eps}$ it follows $||\psi_{\eps}||_{L^{\infty}(\Om_{\eps})}=1$, we get \be\lab{10-9-17-1}\psi_{\eps}(y_{\eps})=1.\ee This in turn implies $y_{\eps}\to\infty$ as $\eps\to 0$. On the other hand, \eqref{nov-2-10} yields that $\psi_{\eps}(y_{\eps})\to 0$. This contradicts \eqref{10-9-17-1}. Hence the  uniqueness result follows.
\hfill$\square$

\vspace{3mm}

\appendix
\numberwithin{equation}{section}

\appendix
\section{}

Define \be \label{hatF}{\hat F}(w)=  \displaystyle\frac{1}{2}\int_{\Rn}\int_{\Rn}\f{|w(x)-w(y)|^2}{|x-y|^{N+2s}} dxdy+  \frac{1}{q+1}\int w^{q+1} dx ,\ee
where $q>p\geq 2^*-1$.
For $\rho>0$, set
$$X_0(\rho\Om):=\{w\in H^s(\Rn): w=0\quad \text{in}\quad \Rn\setminus\rho\Om\},$$
$$N_{\rho}= \bigg\{ w\in X_0(\rho\Om)\cap L^{q+1}(\rho\Om):
 \  \int_{\rho\Om} w^{p+1}dx=1 \bigg\}.$$
Define $$S_{\rho}:=\inf_{w\in N_{\rho}} \hat{F}(w).$$

\begin{theorem}\label{A.1}
(i) If  $p=2^{*}-1$, then  $S_{\rho}\to \frac{\mathcal{S}}{2}$ as $\rho\to\infty$, where  $\mathcal{S}$ is as defined in \eqref{go}.

(ii) If  $p>2^{*}-1$, then $S_{\rho}\to \mathcal{K}$ as $\rho\to\infty$, where $\mathcal{K}$ is as defined in \eqref{a2}.
\end{theorem}
\begin{proof}
{\bf Step $1$}: First we prove that $\lim_{\rho \to \infty}S_{\rho} \leq \frac{S}{2}. $ 
Let us consider the function $U(x)$ defined as in \eqref{ent-U}. We know that $\mathcal{S}$ is achieved by $U$ and $U$ is the unique ground state 
solution of \eqref{ent} with $\displaystyle\int_{\Rn}U^{2^*}(x)dx=1. $

Define $$\displaystyle U_{\rho}(x):=\rho^{-\frac{(N-2s)}{4}}U\left(\f{x}{\sqrt{\rho}}\right) \quad\text{and}\quad
\phi_{\rho}(x)=\phi\left(\f{x}{\rho}\right)$$ where $\phi\in C_0^{\infty}(\Rn), \ \text{supp}(\phi) \in \Om,\  \phi \equiv 1 $ in $\frac{\Om}{2}, \ 0 \leq \phi \leq 1,
\ |\na \phi| \leq \frac{2}{d}$, where $d=diam(\Om)$. It is easy to see that $U_{\rho}$ is also a solution of \eqref{ent}. 

Set $v_{\rho}(x):=U_{\rho}(x)\phi_\rho(x)$ and
$\hat v_{\rho}(x)=\f{v_\rho}{|v_{\rho}|_{L^{2^*}(\rho\Om)}}. $ Then $\hat v_{\rho} \in N_\rho$ and thus, 
\be\lab{sep-11-1}S_\rho\leq \hat{F}(\hat v_{\rho})\ee
Note that,
$$\int_{\rho\Om}v_{\rho}^{2^*}dx=\rho^{-\f{N}{2}}\int_{\rho\Om}U^{2^*}(\f{x}{\sqrt{\rho}})\phi^{2^*}(\f{x}{\rho})dx=\int_{\rho\Om}U^{2^*}(x)\phi^{2^*}(\f{x}{\sqrt{\rho}})dx.$$
Therefore,
\be\lab{i}
\lim_{\rho \to \infty}\int_{\rho\Om}v_{\rho}^{2^*}dx=\int_{\Rn} U^{2^*}(x)dx=1.
\ee
Similarly, 
\begin{align}\label{ii}
\lim_{\rho \to \infty}\int_{\rho\Om}\hat v_{\rho}^{q+1}dx=
 \lim_{\rho \to \infty}\f{\rho^\f{(N+2s)-q(N-2s)}{4}}{|v_{\rho}|^{q+1}_{L^{2^*}(\rho\Om)}}\int_{\Rn}U^{q+1}(x)\phi^{q+1}(\f{x}{\sqrt{\rho}})dx= 0,
\end{align}
as $q>\f{N+2s}{N-2s}$. Hence, from \eqref{sep-11-1},
\be\lab{sep-11-2}
\lim_{\rho \to \infty} S_{\rho}\leq\lim_{\rho \to \infty}\f{1}{2}\int_{\Rn}\int_{\Rn}\frac{|\hat v_\rho(x)-\hat v_\rho(y)|^2}{|x-y|^{N+2s}}dxdy=\lim_{\rho \to \infty}\f{1}{2}\int_{\Rn}\int_{\Rn}\frac{|v_\rho(x)-v_\rho(y)|^2}{|x-y|^{N+2s}}dxdy. 
\ee

Now,
\be\lab{sep-11-2'} \int_{\Rn}\int_{\Rn}\frac{|{v_\rho}(x)-{v_\rho}(y)|^2}{|x-y|^{N+2s}}dxdy= I_\rho^1+I_\rho^2+I_\rho^3, \ee
 where 
 $$I_\rho^1:=\int_{\Rn}\int_{\Rn}\frac{|U_\rho(x)-U_\rho(y)|^2}{|x-y|^{N+2s}}\phi_\rho^2(x)dydx, $$
 \begin{equation}
  I_\rho^2:=\int_{\Rn}\int_{\Rn}\frac{|\phi_\rho(x)-\phi_\rho(y)|^2}{|x-y|^{N+2s}}U_\rho^2(y)dxdy,\no\\
 \end{equation}
 \begin{equation}
   I_\rho^3:=\int_{\Rn}\int_{\Rn}\frac{(U_\rho(x)-U_\rho(y))(\phi_\rho(x)-\phi_\rho(y))U_\rho(y)\phi_\rho(x)}{|x-y|^{N+2s}}dxdy.\no\\
 \end{equation}
 
 A simple calculation yields 
 \bea\lab{sep-11-3}
 \lim_{\rho \to \infty} I_\rho^1&=&\lim_{\rho \to \infty} \int_{\Rn} \int_{\Rn}\frac{|U(x)-U(y)|^2}{|x-y|^{N+2s}}\phi^2(\f{x}{\sqrt{\rho}})dydx\no\\
 &=& \int_{\Rn} \int_{\Rn}\frac{|U(x)-U(y)|^2}{|x-y|^{N+2s}}dydx=\mathcal{S}.
\eea
Using change of variable, it is not difficult to see that
\be\lab{sep-11-4} I_\rho^2=\int_{\Rn} \int_{\Rn} F_{\rho}(x,y)dxdy, \quad\text{where}\quad F_{\rho}(x,y)=\frac{|\phi(\f{x}{\sqrt{\rho}})-\phi(\f{y}{\sqrt{\rho}})|^2U^2(x)}{|x-y|^{N+2s}}.
\ee
Clearly, $F_{\rho}(x,y)\to 0$ pointwise as $\rho\to\infty$. Using dominated convergence theorem, we aim to show that  $\lim_{\rho \to \infty}I_{\rho}^2=0 $. 
Let $$D_1:=\{(x,y) \in \Rn \times \Rn:|x-y| \leq 1\},$$
$$ D_2:=\{(x,y) \in \Rn \times \Rn:|x-y|> 1\}. $$
Thus, 
\be
I_\rho^2=\int_{D_1} F_\rho(x,y)dxdy+\int_{D_2}F_\rho(x,y)dxdy=:I_\rho^{2,1}+I_\rho^{2,2}\no\\
\ee
In $D_1$, we estimate $F_\rho(x,y)$ as follows:
\begin{eqnarray}
F_\rho(x,y)=\frac{|\phi(\f{x}{\sqrt{\rho}})-\phi(\f{y}{\sqrt{\rho}})|^2U^2(x)}{|x-y|^{N+2s}}&\leq& \frac{|\f{x}{\sqrt{\rho}}-\f{y}{\sqrt{\rho}}|^2\norm{\na \phi}_{L^{\infty}(\Rn)}U^2(x)}{|x-y|^{N+2s}}\no\\
 &\leq& \frac{1}{\rho}|x-y|^{2-(N+2s)}\norm{\na \phi}_{L^{\infty}(\Rn)}U^2(x)\no\\
 &\leq& |x-y|^{2-(N+2s)}\norm{\na \phi}_{L^{\infty}(\Rn)}U^2(x),\no 
\end{eqnarray}
for $\rho>1$.
Moreover,
\begin{align}
 &\int_{D_1} |x-y|^{2-(N+2s)}\norm{\na \phi}_{L^{\infty}(\Rn)}U^2(x)dydx\no\\
 &\leq\norm{\na \phi}_{L^{\infty}(\Rn)}\int_{x \in \Rn}U^2(x)\int_{y \in \Rn,\ |x-y| \leq 1}|x-y|^{2-(N+2s)}dydx\no\\
 &=\norm{\na \phi}_{L^{\infty}(\Rn)}\norm{U}^2_{L^2(\Rn)}Nw_N\int_0^1 r^{1-2s}dr <\infty.\no
\end{align}
Hence, by the dominated convergence theorem we see that
$\lim_{\rho \to \infty}I_{\rho}^{2,1}=0. $
On the other hand, in $D_2$ we estimate $F_{\rho}(x,y)$ as follows:
\be\lab{sep-11-5} F_\rho(x,y)\leq\frac{4\norm{\phi}_{L^{\infty}(\Rn)}U^2(x)}{|x-y|^{N+2s}}.\ee
Proceeding same way as above, we can show that RHS of \eqref{sep-11-5} is in $L^{\infty}(D_2)$. Hence, by the dominated convergence theorem we see that
$\lim_{\rho \to \infty}I_{\rho}^{2,2}=0$. Consequently,
\be\lab{sep-11-6}\lim_{\rho \to \infty}I_{\rho}^{2}=0.\ee
Using change of variable, we see that
\be\lab{sep-11-7}
I_{\rho}^3=\int_{\Rn}\int_{\Rn}H_{\rho}(x,y)dxdy,
\ee
where,$$ H_{\rho}(x,y)=\frac{|U(x)-U(y)||\phi(\f{x}{\sqrt{\rho}})-\phi(\f{y}{\sqrt{\rho}})||U(x)||\phi(\f{y}{\sqrt{\rho}})|}{|x-y|^{N+2s}}. $$

Clearly  $H_{\rho}(x,y)\to 0$ pointwise as $\rho\to\infty$. Moreover,
\begin{align}
 &|H_\rho(x,y)| \leq \frac{|U(x)-U(y)||\phi(\f{x}{\sqrt{\rho}})-\phi(\f{y}{\sqrt{\rho}})||U(x)||\phi(\f{y}{\sqrt{\rho}})|}{|x-y|^{N+2s}}\no\\
&\leq\frac{1}{2}\frac{|U(x)-U(y)|^2}{|x-y|^{N+2s}}+\frac{1}{2}\frac{|\phi(\f{x}{\sqrt{\rho}})-\phi(\f{y}{\sqrt{\rho}})|^2U^2(x)}{|x-y|^{N+2s}}\end{align}
The 1st term on RHS is in $L^1(\Rn\times\Rn)$ and 2nd term can be dominated by $L^1$ function as before. Hence by dominated convergence theorem, we have
\be\lab{sep-11-8}\lim_{\rho \to \infty}I_\rho^3(x,y)=\lim_{\rho \to \infty}\int_{\Rn} \int_{\Rn}H_\rho(x,y)dxdy=0.\ee
As a result, combining \eqref{sep-11-3}, \eqref{sep-11-6}, \eqref{sep-11-8}, along with \eqref{sep-11-2'} and \eqref{sep-11-2} we obtain
Hence we obtain that $\lim_{\rho\to\infty}S_{\rho} \leq \f{\mathcal{S}}{2}$.

{\bf Step $2$:} In this step we aim to show $\lim_{\rho\to\infty}S_{\rho} \geq \f{\mathcal{S}}{2}$.  let $\delta>0$ be arbitrary. 
As $S_{\rho}=\inf_{w \in N_\rho} \hat F(w), $  there exists $w_{\rho,\delta} \in N_\rho$ such that
\begin{equation}\label{ii}
 \hat F(w_{\rho,\delta})<S_{\rho}+\delta.
\end{equation}

Let $\eta(.)$ be the standard mollifier function, i.e, $\eta(x)=C\exp(\f{1}{|x|^2-1})$ if $|x|<1$ and $0$ otherwise. Set $\eta_{\sigma}(x)=\sigma^{-N}\eta(\f{x}{\sigma}). $

Define $w_{\rho,\delta}^\sigma:=w_{\rho,\delta}*\eta_{\sigma}$ and $v_{\rho,\delta}^{\sigma}=\frac{w_{\rho,\delta}^\sigma}{|w_{\rho,\delta}^\sigma|_{L^{2^*}(\Rn)} }. $

We note that $v_{\rho,\delta}^\sigma \in C_0^\infty(\Rn) \cap N$
where $$N:=\big\{w \in D^{s,2}(\Rn):  w \in L^{q+1}(\Rn), \int_{\Rn}w^{2^*}dx=1\big\} $$
and $D^{s,2}(\Rn)$ is completion of $C_0^\infty(\Rn)$ with the norm
$\displaystyle\left(\int_{\Rn}\int_{\Rn}\frac{|u(x)-u(y)|^2}{|x-y|^{N+2s}}dxdy\right)^\f{1}{2}$.

Note that $v_{\rho,\delta}^\sigma \to w_{\rho,\delta}$ in $D^{s,2}(\Rn)\cap L^{q+1}(\Rn)$ as $\sigma \to 0. $

Hence, we have, $$\f{\mathcal{S}}{2} \leq \hat F(v_{\rho,\delta}^\sigma) \to \hat F(w_{\rho,\delta}). $$
Combining this with \eqref{ii} we have, $\f{\mathcal{S}}{2} < S_{\rho}+ \delta. $
As $\delta>0$ is arbitrary we have, $\f{\mathcal{S}}{2} \leq S_{\rho}. $
This implies, $\f{\mathcal{S}}{2} \leq \lim_{\rho \to \infty}S_{\rho}. $ This completes the proof. 

\noi{\bf Second part}:

Let $w \in D^{s,2}(\Rn)\cap L^{q+1}(\Rn)$ be a minimizer for $\mathcal{K}$ (existence is guaranteed by \cite[Theorem 1.4]{BM}) with $\displaystyle\int_{\Rn}w^{p+1}dx=1$. 

Define $\phi_{\rho}$ as in step $1. $
Set $w_{\rho}:=w\phi_{\rho}$ and $\hat w_{\rho}=\f{w_\rho}{|w_\rho|_{L^{p+1}(\Rn)}}. $
Then $\hat w_{\rho} \in N_{\rho}$. Consequently, $S_{\rho} \leq \hat F(\hat w_\rho). $ Proceeding before as in step $1, $ we can show that
$\hat F(\hat w_\rho) \to \mathcal{K}$ as $\rho \to \infty. $ Hence, $\lim_{\rho \to \infty}S_{\rho} \leq \mathcal{K}$. To get the other sided inequality, we use the same idea as first part.  Hence, the result follows. 
\end{proof}
 
\vspace{3mm}

\noi{\bf Acknowledgement}: The first author is supported by the INSPIRE research grant
DST/INSPIRE 04/2013/000152 and the second author is supported by the
NBHM grant 2/39(12)/2014/RD-II. The third author acknowledges funding from LMAP UMR CNRS 5142, Universit\'e de Pau et des Pays de l'Adour.

\end{document}